\documentclass[12pt]{amsart}

\newcommand{\indentalign}{\hspace{0.3in}&\hspace{-0.3in}}
\newcommand{\la}{\langle}
\newcommand{\ra}{\rangle}
\renewcommand{\Re}{\operatorname{Re}}
\renewcommand{\Im}{\operatorname{Im}}
\newcommand{\sech}{\operatorname{sech}}
\newcommand{\defeq}{\stackrel{\rm{def}}{=}}

\newcommand{\cR}{\mathbb{R}}

\newcommand{\cF}{\mathcal{F}}
\newcommand{\ds}{\displaystyle}
\newcommand{\supp}{\mathrm{supp}\,}

\newcommand{\cH}{\mathcal{H}}

\usepackage{amssymb}
\usepackage{amsthm}
\usepackage{amsxtra}
\usepackage{graphicx}
\usepackage{color}

\newtheorem{theorem}{Theorem}[section]

\newtheorem{proposition}{Proposition}[section]
\newtheorem{lemma}[proposition]{Lemma}

\theoremstyle{remark}

\newtheorem*{remarks}{Remarks}

\numberwithin{equation}{section}

\title[3d NLS contracting sphere blow-up]{A solution to the focusing 3d NLS that blows up on a contracting sphere}
\author{Justin Holmer}
\address{Brown University, Providence, RI, USA}
\email{holmer@math.brown.edu}
\author{Galina Perelman}
\address{Universit\'e Paris-Est Cr\'eteil, Cr\'eteil Cedex, France}
\email{galina.perelman@u-pec.fr}
\author{Svetlana Roudenko}
\address{The George Washington University, Washington, DC, USA}
\email{roudenko@gwu.edu}
\thanks{to appear in Transactions of the AMS}

\begin{document}

\maketitle

\begin{abstract}
We rigorously construct radial $H^1$ solutions to the 3d cubic focusing NLS equation $i\partial_t \psi + \Delta \psi + 2 |\psi|^2\psi=0$ that blow-up along a contracting sphere.  With blow-up time set to $t=0$, the solutions concentrate on a sphere at radius $\sim t^{1/3}$ but focus towards this sphere at the  faster rate $\sim t^{2/3}$.   Such dynamics were originally proposed heuristically by Degtyarev-Zakharov-Rudakov \cite{DZR} in 1975 and independently later in Holmer-Roudenko \cite{AMRX} in 2007, where it was demonstrated to be consistent with all conservation laws of this equation.  In the latter paper, it was proposed as a solution that would yield divergence of the $L_x^3$ norm within the ``wide'' radius $\sim \|\nabla u(t)\|_{L_x^2}^{-1/2}$ but not within the ``tight'' radius $\sim \|\nabla u(t)\|_{L_x^2}^{-2}$, the second being the rate of contraction of self-similar blow-up solutions observed numerically and described in detail in Sulem-Sulem \cite[Chapter 7]{SS}.
\end{abstract}

\section{Introduction}

Consider the 3d cubic focusing NLS equation on $\mathbb{R}^3$:
\begin{equation}
\label{E:NLS-1}
i\partial_t \psi + \Delta \psi + 2|\psi|^2\psi=0,
\end{equation}
where $\psi=\psi(x,t)\in \mathbb{C}$ and $x\in \mathbb{R}^3$, $t \in \mathbb{R}$.  The initial-value problem posed with initial-data in $H_x^1$ is locally well-posed and there is a unique solution in $C([0,T); H_x^1)$ on a maximal forward life-span $[0,T)$.  If $T<\infty$, then $\lim_{t\nearrow T} \|\nabla \psi(t) \|_{L_x^2} = \infty$.

During their lifespan,
the solutions $\psi$ to \eqref{E:NLS-1} satisfy mass, energy  and momentum conservation laws.
\begin{equation}
\label{E:MASS}
M(\psi(t)) \equiv \int |\psi(x,t)|^2 dx=\int |\psi_0(x)|^2 dx = M(\psi_0),
\end{equation}

\begin{equation}
\label{E:ENERGY}
E(\psi (t))\equiv\int (|\nabla \psi(x,t)|^2 - |\psi(x,t)|^4)\, dx=E(\psi_0),
\end{equation}

\begin{equation}
\label{E:MOMENTUM}
P(\psi(t))\equiv \int (\bar \psi \nabla \psi-\psi \nabla\bar \psi)\, dx=P(\psi_0).
\end{equation}

The case $T<\infty$ occurs for a large class of initial data (see the discussion in Holmer-Platte-Roudenko \cite{HR-above}) and in this case we say that the solution \emph{blows-up in finite time}.  Numerical results (see Sulem-Sulem \cite[Chapter 7]{SS}) describe the existence of self-similar radial blow-up solutions of the form
\begin{equation}
 \label{E:atorigin}
u(x,t) \approx \frac{1}{\lambda(t)}U\left( \frac{x}{\lambda(t)}
\right) e^{i\log(T-t)} \text{ with }\lambda(t)=\sqrt{2b(T-t)},
\end{equation}
where $U=U(x)$ is a stationary profile
satisfying the nonlinear elliptic equation
\begin{equation}
\label{E:usual-profile}
\Delta U - U + i b ( U + y\cdot\nabla U) +2 |U|^2 U =0.
\end{equation}
Numerics suggest that $b>0$ can be selected so that a nontrivial \emph{zero-energy} solution $U$ to \eqref{E:usual-profile} exists.  Asymptotics as $|x|\to \infty$ show that any solution of \eqref{E:usual-profile} fails to belong to $\dot H_x^{1/2}$ (and hence $H_x^1$).  Nevertheless, it is expected that $H_x^1$ blow-up solutions to \eqref{E:NLS-1} well-modeled by \eqref{E:atorigin} \emph{near the origin} exist, and the proof remains an important open problem.

In this paper we consider only radial solutions, thus, $P(\psi) = 0$.
We construct a family of finite-time blow-up solutions to \eqref{E:NLS-1} with different dynamics -- they focus toward a sphere, while at the same time the radius of sphere is shrinking.
We study the problem by converting it (via the time-reversal, $u(t) \mapsto \bar u(t)$,
and time-translation symmetry of \eqref{E:NLS-1}) to one with solutions defined on $(0,t_0]$
for $t_0>0$ that ``start'' at $t_0>0$ and blow-up at $t=0$ as they evolve backward in time.

Our main result is the following:

\begin{theorem}
\label{T:MainTheorem}
For any $e\in \cR$ there
exists a radial solution $\psi \in C((0,t_0],H^1(\cR^3))$ of \eqref{E:NLS-1}
with $E(\psi)=e$,
which starts at time $t_0 > 0$ and, while evolving backwards in time, blows-up at $t=0$, and has the following form:
\begin{equation}
\label{E:Psi-1}
\psi(r,t) = e^{i\theta(t)} e^{iv(t)r/2} \lambda(t) \, \varphi\big(\lambda(t)(r-q(t)) \big) + h(r,t), \quad r = |x|, x \in \cR^3,
\end{equation}
where $\varphi (r) = \sech (r)$ and the time-dependent parameters $(q, v, \lambda, \theta)$ satisfy
\begin{equation}
\label{E:params-0}
\begin{aligned}
& q(t) = q_0 t^{1/3}               &\qquad&  q_0 = 2^{1/3}3^{1/6}\\
& \lambda(t) = \lambda_0 t^{-2/3}  && \lambda_0=2^{-2/3}3^{-1/3} \\
& v(t) = v_0 t^{-2/3}              &&  v_0= 2^{1/3}3^{-5/6} \\
& \theta(t) = \theta_0 t^{-1/3}    && \theta_0= 2^{-1/3}3^{-2/3}
\end{aligned}
\end{equation}
and
\begin{equation}
\label{E:h-est}
\|h(t)\|_{L^2(\cR^3)}+t^{2/3}\|h(t)\|_{H^1(\cR^3)} +t^{-1/3}\||x|\,h(t)\|_{L^2(\cR^3)} \leq C \, t^{1/3}.
\end{equation}
Note that it follows from \eqref{E:params-0} and \eqref{E:h-est} that
$$
\|\nabla \psi(t)\|_{L^2(\mathbb{R}^3)} \sim t^{-2/3} \quad \text{ and } \quad \||x|\,\psi(t)\|_{L^2(\cR^3)}\sim t^{1/3}.
$$
\end{theorem}

\begin{remarks}

1.  Let $\varphi$ be a solution of the nonlinear elliptic equation
\begin{equation}
\label{E:groundstate}
-\varphi + \Delta \varphi + 2 |\varphi|^2 \varphi = 0.
\end{equation}
The unique (up to translation) minimal mass $H^1$ solution of \eqref{E:groundstate}
is called the {\it ground-state}. It is smooth, radial, real-valued and positive, and exponentially decaying (see,
for example, Tao \cite[Apx. B]{Tao-book}). Note that $\psi (x,t) = e^{it} \varphi (x)$ is a solution of \eqref{E:NLS-1}. In one dimension, i.e., $d = 1$, the ground state of \eqref{E:groundstate} is explicit, namely, $\varphi(x) = \sech(x)$. It is this function we use for the profile of the blow-up in \eqref{E:Psi-1} (and not the 3d ground state solution of \eqref{E:groundstate}).

2. It follows from \eqref{E:params-0} that the solutions given by Theorem \ref{T:MainTheorem}
satisfy $\|\psi\|_{L^2(\cR^3)}=\|\varphi\|_{L^2(\cR)}= \sqrt 2$. Applying scaling, one then obtains
contracting sphere blow up solutions with an arbitrary mass.

3. We do not address in this note the questions of stability of these solutions (the numerical and heuristical results of \cite{FGW}  indicate that the behavior \eqref{E:Psi-1}, \eqref{E:params-0} should be stable with respect to radial perturbations). Still, let  us mention that  the arguments developed in the proof of Theorem \ref{T:MainTheorem}
can be easily adjusted  in order to produce  blow up solutions of the form
$$
\psi(t)=S(t)+f+o_{H^1}(1), \quad t\rightarrow 0,
$$
where $S (t)$ is a solution of \eqref{E:NLS-1} given by Theorem \ref{T:MainTheorem},
and $f$ is an arbitrarily smooth radial decaying function vanishing to a sufficiently high order
at the origin, very much in the spirit of \cite{BW}.

4. A numerical visualization of the blow-up described by Theorem \ref{T:MainTheorem} is given in Figure 1, courtesy of R. Platte, Arizona State University.
\begin{figure}[h]
\begin{center}
\includegraphics[width=2.7in]{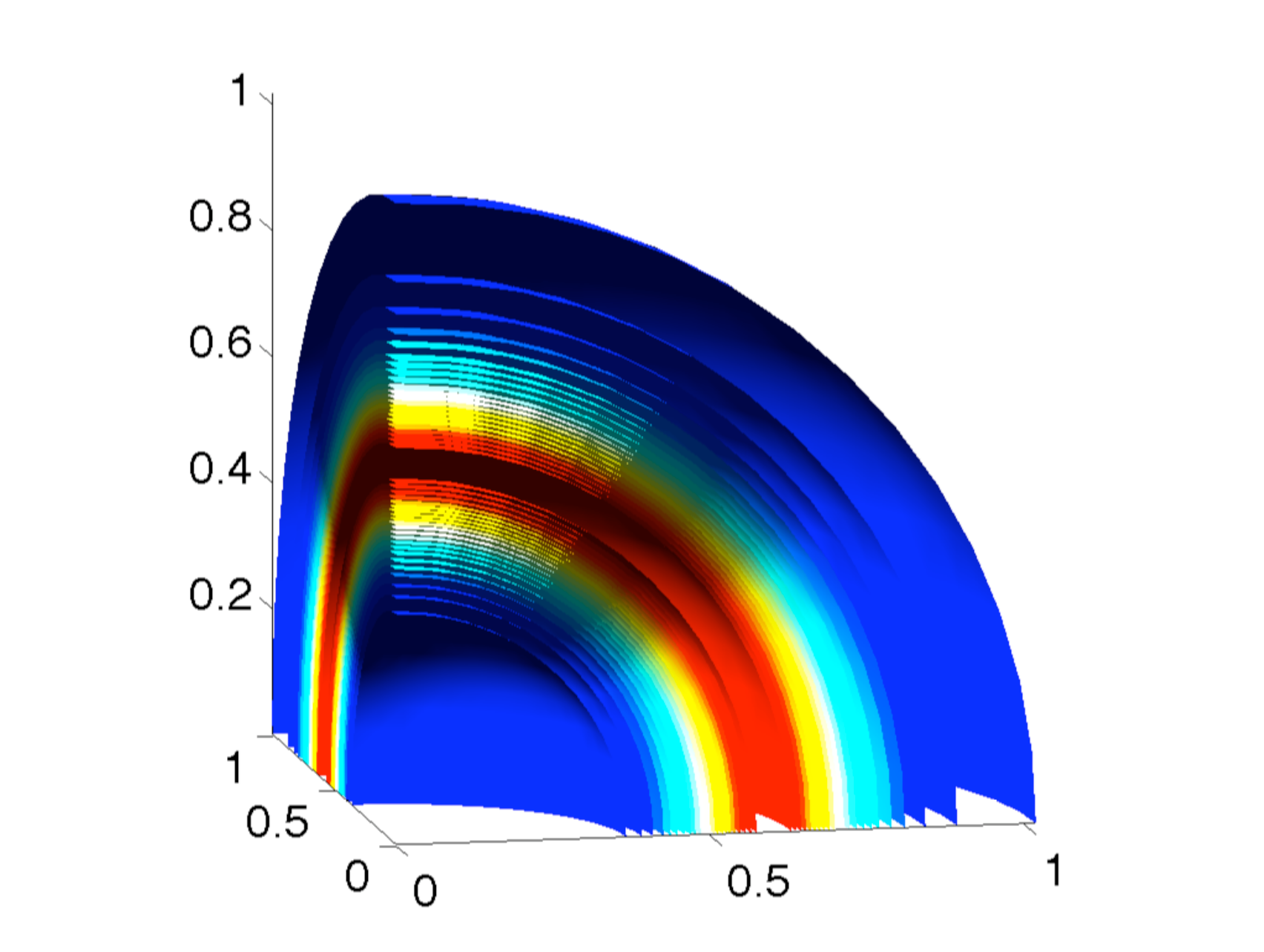}\hspace{-1.5cm}\includegraphics[width=2.7in]{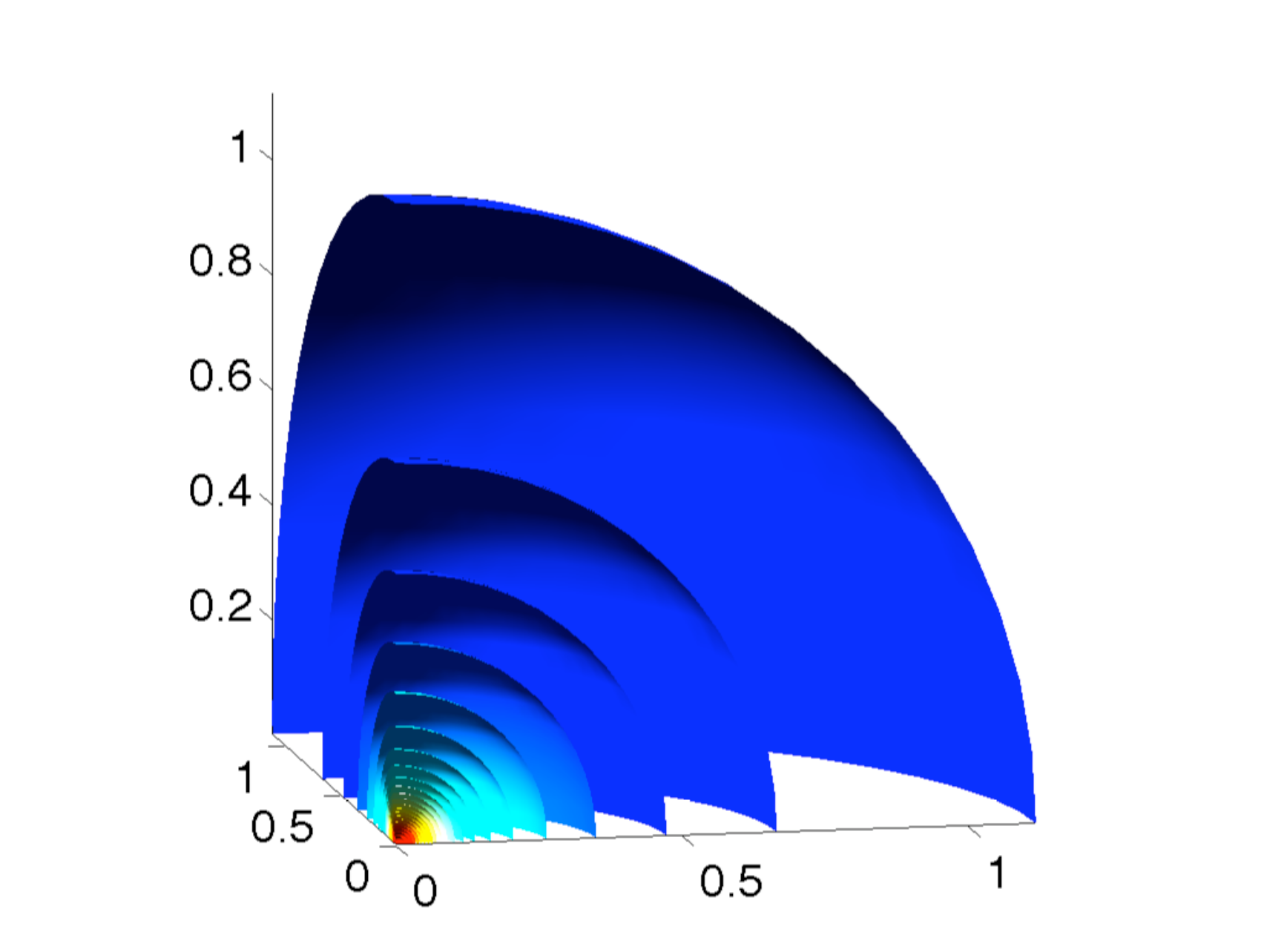}
\end{center}
\caption{\small{A contracting sphere collapse in the 3d cubic NLS: (a) a concentrating layer of the collapsing solution is around $r\sim 0.6$, (b) the eventual collapse at the origin. Courtesy of R. Platte, Arizona State University.}}
 \label{fig:collapse}
\end{figure}
\smallskip

\end{remarks}

The paper is organized as follows:    In \S \ref{S:construction}, we construct (Prop. \ref{P:1}) for arbitrary $N\in \mathbb{N}$, an approximate solution $\psi^{(N)}$ satisfying
\begin{equation}
\label{E:intro-stuff1}
i\partial_t \psi^{(N)} + \Delta \psi^{(N)} + 2|\psi^{(N)}|^2 \psi^{(N)} = iH^{(N)},
\end{equation}
where
$$
t^{2k/3}\|\nabla^k H^{(N)}(t)\|_{L^2(\mathbb{R}^3)} \lesssim_k t^{(2N-1)/3} \,,\qquad 0< t \leq t_0(N),
$$
with implicit constant independent of $N$.    To accomplish this, we set
\begin{equation}
\label{E:intro-stuff4}
\psi^{(N)}(r,t) = e^{i\theta} e^{ivr/2} \lambda \,U^{(N)}(\rho, t),
\end{equation}
where $\rho = \lambda(r-q)$ is the recentered and rescaled spatial variable and $q$, $v$, $\theta$, $\lambda$ are $t$-dependent parameters.  We then seek to suitably define the parameters and profile $U^{(N)}$.  We impose the \emph{a priori} conditions
$$
\left\{
\begin{aligned}
& v = q' \\
& \theta'= - \tfrac14 v^2 - \tfrac12 v' q + \lambda^2,
\end{aligned}
\right.
$$
which achieve convenient simplifications of the equation that arises when \eqref{E:intro-stuff4} is substituted into \eqref{E:intro-stuff1}, and enable us to solve for $v$ and $\theta$, once $q$ and $\lambda$ have been specified.  For additional convenience, we set $\omega = \lambda^{-1}q^{-2}$ (effectively replacing $\lambda$ by $\omega$) and\footnote{In fact, once $\chi$ is constructed, we take $U^{(N)}$ to be a \emph{cutoff} of $\varphi + \chi$.} $\chi = U^{(N)} - \varphi$ (recall $\varphi(\rho) = \sech\rho$).  Then \eqref{E:intro-stuff1} will follow from
\begin{equation}
\label{E:intro-stuff2}
| \partial_\rho^\ell (\mathcal{H} \vec \chi + \ell(q,\omega)\vec\chi + \mathcal{F}(q, \omega) + F^{\geq 2}(\vec\chi))| \leq C_{\ell, \mu} \,t^{(2N+3)/3} e^{-\mu |\rho|} \,, \quad 0<\mu<1,
\end{equation}
where $\mathcal{H}$ is the 1D cubic linearized Hamiltonian -- $\mathcal{H}$, $\ell$, $\mathcal{F}$, and $F^{\geq 2}$ are defined precisely in \S \ref{S:construction}.  We then seek a solution to \eqref{E:intro-stuff2} by expanding
\begin{equation}
\label{E:intro-stuff3}
\omega(t)= \sum_{k=0}^{N} \omega_kt^{2k/3} \,, \qquad q(t)=\sum_{k=0}^{N} q_kt^{(2k+1)/3} \,, \qquad \chi(\rho,t)=\sum_{k=1}^{2N+2} t^{k/3}\chi_k(\rho) \,.
\end{equation}
Substituting the expansions \eqref{E:intro-stuff3} into \eqref{E:intro-stuff2}, we can organize terms by their order in $t$ and carry out a recursive construction.  The parameter coefficients $q_k$ and $\omega_k$ are selected at each stage to satisfy solvability conditions for $\chi_{k+1}$.

In \S \ref{S:comparison}, we consider the exact solution $\psi$ with data $\psi|_{t=\epsilon} = \psi^{(N)}(\epsilon)$ for any $0<\epsilon \leq t_0(N)$, and prove (Prop. \ref{P:Cauchy-N}) that for $N$ sufficiently large, the comparison estimate
\begin{equation}
\label{E:intro1}
\sum_{k=0}^2 t^{2k/3}\| \nabla^k(\psi(t) - \psi^{(N)}(t)) \|_{L^2(\mathbb{R}^3)} \lesssim t^{2N/3} \,, \qquad \epsilon \leq t \leq t_0(N),
\end{equation}
and localization bound (Prop. \ref{P:H2-variance})
\begin{equation}
\label{E:intro2}
\| x \psi(t) \|_{L^2(\mathbb{R}^3)} \lesssim t^{1/3}
\end{equation}
hold, independently of $0<\epsilon \leq t_0(N)$.
The estimates \eqref{E:intro1} and \eqref{E:intro2} are proved using a Lyapunov functional $G(h)$ (see \eqref{E:G-def}), built from energy, mass, and radial momentum.

Finally, in \S 4, we complete the proof of Theorem \ref{T:MainTheorem} using estimates \eqref{E:intro1}, \eqref{E:intro2} and a compactness argument.
\smallskip

This result was announced and described in detail \cite{G2011} by G. Perelman in the  Analysis seminar at Universit\'e de Cergy-Pontoise in December 2011.  In February 2012, Merle, Rapha\"el, \& Szeftel posted a paper \cite{MRS} on arxiv.org proving a closely related result.  We stress that our result was obtained independently of \cite{MRS}. 

\subsection{Acknowledgements}
J.H. was supported in part by NSF grant DMS-0901582 and a Sloan Research Fellowship (BR-4919). S.R. was supported in part by NSF grant DMS-1103274.  This work was also partially supported by a grant from the Simons Foundation \# 209839 to S.R.

\section{Construction of an approximate solution}
\label{S:construction}

We start with writing a solution to \eqref{E:NLS-1} as
\begin{equation}
\label{E:Psi-2}
\psi(r,t) = e^{i\theta(t)} e^{iv(t)r/2} \lambda(t) \, U \big(\lambda(t)(r-q(t)), t \big).
\end{equation}
Substituting it into \eqref{E:NLS-1} and introducing a rescaled coordinate
$$
\rho = \lambda (r-q),
$$
we obtain that $U = U(\rho, t)$ solves
\begin{multline}
\label{E:unforced-U}
\qquad i \, \frac1{\lambda^2} \, U_t + U_{\rho \rho} - \frac1{\lambda^2} \left( \theta^\prime + \frac{v^2}{4} + \frac{v^\prime \, q}{2} + \frac{v^\prime \rho}{2\lambda}\right) U + 2\,  |U|^2 U  + i \left(\frac{v}{\lambda} - \frac{q^\prime}{\lambda} \right) U_\rho\qquad\\
= - \frac2{\rho + \lambda\, q} \, U_\rho
-i \frac{\lambda^\prime}{\lambda^3} \big(1 + \rho \, \partial_\rho \big) \, U - i \frac{v}{\lambda (\rho + \lambda \,q)} \, U.
\end{multline}
Imposing the following two conditions
\label{E:params}
\begin{equation}
\label{E: params-1}
\begin{aligned}
&v(t)  = q^\prime(t),\\
&\theta^\prime (t) + \frac{v(t)^2}{4} + \frac{v^\prime(t) \, q(t)}{2}  =  \lambda(t)^2,
\end{aligned}
\end{equation}
we obtain
\begin{multline}
\label{E:EqU}
\qquad i \, \frac1{\lambda^2} \, U_t + U_{\rho \rho} - U + 2\,  |U|^2 U  = - \frac2{\rho + \lambda\, q} \, U_\rho\\ + \frac{q^{\prime\prime} \rho}{2 \lambda^3} \, U
-i \frac{\lambda^\prime}{\lambda^3} \big(1 + \rho \, \partial_\rho \big) \, U - i \frac{q^\prime}{\lambda (\rho + \lambda \,q)} \, U. \qquad
\end{multline}
We next denote
\begin{equation}
\label{E: params-2}
\ds {\omega} = \frac1{\lambda \, q^2}
\end{equation}
and rewrite \eqref{E:EqU} as
\begin{multline}
\label{E:EqU-E}
\qquad i \, q^4 \omega^2 \, U_t + U_{\rho \rho} - U + 2\,  |U|^2 U  = - \frac{2 \, q \omega}{\rho \omega q + 1} \, U_\rho\\
+ \frac12 q^{\prime\prime} q^6 \omega^3 \rho \, U
+ i (2 q^\prime q^3\omega^2 + \omega^\prime \omega q^4) \big(1 + \rho \, \partial_\rho \big) \, U - i \frac{q^\prime q^3 \omega^2}{\rho \omega q +1} \, U.
\end{multline}

Define the Pauli matrices by
$$
\sigma_1 = \left(\begin{array}{cc}0&1 \\1&0 \end{array}\right) \quad \text{and} \quad \sigma_3 = \left(\begin{array}{cc}1&0 \\0&-1 \end{array}\right) \,.
$$
Split $U$ into the localized part $\varphi(\rho)=\sech \rho$ and the external part $\chi (\rho, t)$:
\begin{equation}
\label{E:split-U}
U(\rho, t) = \varphi(\rho) + \chi (\rho, t).
\end{equation}
Substituting \eqref{E:split-U} into \eqref{E:EqU-E} and using the vector notation
$\ds \vec{\chi} = \left(\begin{array}{c} \chi\\ \bar{\chi} \end{array}\right)$,
we obtain the following equation for $\vec{\chi}$:
\begin{equation}
\label{E:chi-vector}
\cH \vec{\chi} + l(q,\omega) \vec{\chi} + \cF (q,\omega) + F^{\geq 2} (\vec \chi)=0,
\end{equation}
where the components of the left-hand side are defined as follows:

$\bullet$ the Hamiltonian corresponding to the linearization of the 1D cubic NLS $i\partial_t u + \partial_\rho^2 u +2|u|^2u=0$ around the soliton $u(t,\rho) = e^{it}\sech \rho$:
\begin{equation*}
\cH = (-\partial_\rho^2 +1)\, \sigma_3 - 4 \varphi^2(\rho) \sigma_3 - 2 \varphi^2(\rho) \sigma_3 \sigma_1 \,,
\end{equation*}

$\bullet$ the linear operator
\begin{equation*}
\begin{split}
&l(q,\omega) =- i q^4 \omega^2\partial_t-\frac{2q \omega}{\rho \omega q + 1} \partial_\rho \sigma_3 + i (2 q^\prime q^3 \omega^2 + \omega^\prime \omega q^4) (1 + \rho \partial_\rho) - i \frac{q^\prime q^3 \omega^2}{\rho \omega q +1} \\
&+ \frac12 q^{\prime\prime} q^6 \omega^3 \rho \sigma_3,
\end{split}
\end{equation*}

$\bullet$ the terms independent of $\chi$:
$$
\cF(q,\omega) = l(q,\omega)\varphi \left(\begin{array}{c}1\\1\end{array} \right),
$$
and lastly,

$\bullet$ the higher order terms grouped into
$$
F^{\geq 2} ({\vec \chi}) = -4 \varphi |\chi|^2 \left(\begin{array}{r}1\\-1\end{array} \right) - 2 \varphi \left(\begin{array}{r}\chi^2\\ -\bar{\chi}^2\end{array} \right) - 2 |\chi|^2 \left(\begin{array}{r} \chi\\-\bar{\chi}\end{array} \right)\,.
$$

\begin{proposition}
\label{P:1}
There exist  coefficients $(\omega_k)_{k\geq 0}$, $(q_k)_{k\geq 0}$, with $q_2$ arbitrary and $\omega_0=1$, and smooth exponentially decaying functions
$(\chi_k(\rho))_{k\geq 1}$,
\begin{equation}
\label{P1:1}
|e^{\mu |\rho|}\partial_\rho^l\chi_k(\rho)|\leq C_{l,\mu, k},\quad \rho\in \cR,
\end{equation}
for any $k\geq 1$, $l\geq 0$, and $0\leq \mu<1$, verifying the orthogonality conditions
\begin{equation}
\label{P1:2}
(\Re \chi_k, \varphi_\rho)_{L^2(\cR)}=(\Im \chi_k, \varphi)_{L^2(\cR)}=0, \quad \forall k\geq 1,
\end{equation}
such that the following holds.
For any $N\geq 2$, if we define
$$
\omega(t)\equiv \omega^{(N)} (t)\defeq 1+ \sum_{k=1}^{N} \omega_kt^{2k/3}, \qquad  q(t)\equiv q^{(N)}(t) \defeq  \sum_{k=0}^{N} q_kt^{(2k+1)/3},
$$
$$
\text{and} \quad \chi(\rho, t)\equiv \chi^{(N)}(\rho,t) \defeq \sum_{k=1}^{2N+2} t^{k/3}\chi_k(\rho),  \qquad
$$
then \eqref{E:chi-vector} approximately holds:
\begin{equation}
\label{P1:3}
\left|\partial_\rho^l \big(  \cH \vec{\chi} + l(q,\omega) \vec{\chi} + \cF (q,\omega) + F^{\geq 2} (\vec{\chi}) \big) \right| \leq C_{l,\mu} t^{(2N+3)/3}e^{-\mu |\rho|},
\end{equation}
for $0\leq \mu < 1$, $\rho> -\frac12q_0^{-1}t^{-1/3},\, 0< t \leq t_0(N)$, with some $t_0(N)>0$ and with $C_{\ell,\mu}$ independent of $N$.
\end{proposition}

\begin{proof}
Writing
$\ds \omega(t)= \sum_{k=0}^{N} \omega_k t^{2k/3}$, $\ds q(t)=\sum_{k=0}^{N} q_k t^{(2k+1)/3}$,
$\ds \chi(\rho,t)=\sum_{k=1}^{2N+2} t^{k/3}\chi_k(\rho)$, $\omega_0=1$ and   substituting this ansatz into
the expression $ \cH \vec{\chi} + l(q,\omega) \vec{\chi} + \cF (q,\omega) + F^{\geq 2} (\vec{\chi})$,
 we get
\begin{equation}
\label{P1:4}
\begin{split}
& \cH \vec{\chi} + l(q,\omega) \vec{\chi} + \cF (q,\omega) + F^{\geq 2} (\vec{\chi})=
\sum_{k=1}^{2N+2}t^{k/3}(\cH\vec \chi_k-\vec {D_k})+S(t), \\
&S(t)=-\sum_{k\geq 2N+3}t^{k/3}\vec {D_k}+l_1(q,\omega) ( \varphi \left(\begin{array}{c}1\\1\end{array} \right)+\vec \chi),
\end{split}
\end{equation}
where
$$
l_1(q,\omega) = \frac{2\, q\omega(q \omega\rho)^{2N+3}}{\rho \omega q + 1} \partial_\rho \, \sigma_3 + i \frac{q^\prime q^{3} \omega^2(q\omega\rho)^{2N+3}}{\rho \omega q +1},
$$
$$
\vec D_k= {D_k\choose -\bar D_k},\quad D_k= D_k^{(0)}+D_k^{(1)}+D_k^{(2)},
$$
$D_k^{(0)}$, $D_k^{(1)}$, $D_k^{(2)}$ being contributions of $\cF (q,\omega) -l_1(q,\omega) \varphi \left(\begin{array}{c}1\\1\end{array} \right)$,
$(l(q,\omega) -l_1(q,\omega))\vec{\chi}$ and $F^{\geq 2} (\vec{\chi})$ respectively:
\begin{equation*}
\begin{split}
&\cF (q,\omega)-l_1(q,\omega) \varphi \left(\begin{array}{c}1\\1\end{array} \right)=
-\sum_{k\geq 1}t^{k/3} {\vec D^{(0)}_k}(\rho),\\
&(l(q,\omega)-l_1(q,\omega)) \vec{\chi}=-\sum_{k\geq 2}t^{k/3} {\vec D^{(1)}_k}(\rho),\\
&F^{\geq 2} (\vec{\chi})=-\sum_{k\geq 2}t^{k/3} {\vec D^{(2)}_k}(\rho),
\end{split}
\end{equation*}
$\vec D_k^{(i)}= {D_k^{(i)}\choose -\bar D_k^{(i)}}$. Note that  all these sums contain only finite number of terms.

The following structural properties of $D^{(i)}_k$ will be important for our analysis.

\noindent (i)  {\it Dependence on $q_j,\, \omega_j\, \chi_j$}:
$D_k^{(0)} $ depends on $q_j$, $\omega_j$, $j\leq (k-1)/2$, and
$D_k^{(i)}$, $i=1,2$, depend on $q_j$, $\omega_j$, $j\leq (k-2)/2$,
and on $\chi_p$, $p\leq k-1$, only:
\begin{equation*}\begin{split}
&D_k^{(0)}=D_k^{(0)}(\rho;  q_j, \omega_j, j\leq (k-1)/2),\\
&D_k^{(i)}=D_k^{(i)}(\rho;  q_j, \omega_j, j\leq (k-2)/2;\chi_p, p\leq k-1),\quad i=1,2.
\end{split}\end{equation*}
In addition, one has
\begin{equation}
\label{P1:D0}
t^{1/3}D_1^{(0)}=2q_0(t)\varphi_\rho-iq_0^\prime(t)q_0^3(t)(\varphi+2\rho\varphi_\rho)
-\frac12q_0^{\prime\prime}(t)q_0^6(t)\rho\varphi,\quad q_0(t)=q_0t^{1/3},
\end{equation}
and for $k=2l+1$, $l\geq 1$:
\begin{equation}
\label{P1:D1}
D_{2l+1}^{(0)}=D_{2l+1}^{(0,0)}+D_{2l+1}^{(0,1)},
\end{equation}
where $D_{2l+1}^{(0,1)}=D_{2l+1}^{(0,1)}(\rho;  q_j, \omega_j, j\leq l-1)$ depends on $q_j$, $\omega_j$, $j\leq l-1$, only
and $D_{2l+1}^{(0,0)}$ is given by
\begin{equation}
\label{P1:D2}
\begin{split}
&t^{(2l+1)/3}D_{2l+1}^{(0,0)}=\\
&2(q_l(t)+q_0(t)\omega_l(t))\varphi_\rho- \frac12 (q_0^6(t)q_l^{\prime\prime}(t)
+6q_0^{\prime\prime}(t)q_0^5(t)q_l(t)+3q_0^{\prime\prime}(t)q_0^6(t)\omega_l(t))\rho\varphi\\
&-i(q_0^3(t)q_l^{\prime}(t)+3q_0^\prime(t)q_0^2(t)q_l(t)+2q_0^\prime(t)q_0^3(t)\omega_l(t))
(\varphi+2\rho\varphi_\rho)
-iq_0^4(t)\omega_l^\prime(t)(\varphi+\rho\varphi_\rho),
\end{split}
\end{equation}
$q_l(t)=q_lt^{(2l+1)/3},\,\omega_l(t)=\omega_lt^{2l/3}$.

\noindent (ii) {\it Smoothness and decay}: if all $\chi_j,\, j\leq k-1$, verify
\eqref{P1:1}, then a similar estimate holds for $D_k$:
\begin{equation}
\label{P1:SD}
|e^{\mu |\rho|}\partial_\rho^lD_k(\rho)|\leq C_{l,\mu, k},\quad \rho\in \cR,
\end{equation}
for any $l\geq 0$ and $\mu<1$.

\noindent (iii) {\it Parity}:
if for any $j\leq k-1$,
$v_j\defeq\Re \chi_j$ and $u_j\defeq\Im \chi_j$ satisfy
\begin{equation}
\label{P1:S1}
v_j(-\rho)=(-1)^jv_j(\rho),\quad u_j(-\rho)=(-1)^{j+1}u_j(\rho),
\end{equation}
then  one has the same property for ${\mathcal G}^+_k\defeq \Re D_k$, ${\mathcal G}_k^-\defeq \Im D_k$ :
\begin{equation}
\label{P1:S2}
{\mathcal G}_k^+(-\rho)=(-1)^{k}{\mathcal G}_k^+(-\rho),\quad
{\mathcal G}_k^-(-\rho)=(-1)^{k+1}{\mathcal G }_k^-(-\rho).
\end{equation}

Our goal now is to show that we can find coefficients $q_0,q_1,\dots, q_{N}$,
$\omega_1,\dots, \omega_{N}$
($q_0$ being as in Theorem \ref{T:MainTheorem} and $q_2$ arbitrary)
and functions $\chi_1,\dots, \chi_{2N+2}$ satisfying \eqref{P1:1},  \eqref{P1:2}, \eqref{P1:S1}
such that
\begin{equation}
\label{P1:5}
\cH\vec \chi_k=\vec D_k,\quad k=1,\dots, 2N+2.
\end{equation}
Estimate \eqref{P1:3} will then follow from \eqref{P1:4}.
As we see below, the equations with $k=5$ require some special care, therefore, we will go explicitly
through $k=1,\dots, 5$ and then we will proceed by induction.

It will be convenient for us to rewrite \eqref{P1:5} as
\begin{equation}
\label{P1:6}
L_+v_k={\mathcal G}^+_k,\quad L_-u_k={\mathcal G}_k^-,\quad k=1,\dots, 2N+2,
\end{equation}
where $L_+= -\partial_\rho^2+1-6\varphi ^2$,
$L_-=-\partial_\rho^2+1-2\varphi ^2$ are  self-adjoint operators in $L^2(\cR)$
with the spectrum $\sigma(L_+)=\{-3 ,0\}\cup [1,\infty)$, $\sigma(L_-)=\{0\}\cup [1,\infty)$.
One has
$$
L_-\varphi=0,\quad L_+\varphi_\rho=0.
$$
For $k=1$ the equations \eqref{P1:6} give
\begin{equation}
\label{P1: k=1}
L_+v_1={\mathcal G}^+_1,\quad L_-u_1={\mathcal G}^-_1,
\end{equation}
with
$$
{\mathcal G}_1^+=2q_0\partial_\rho\varphi+\frac{1}{9}q_0^7\rho\varphi,
\quad {\mathcal G}_1^-=-\frac13 q_0^4(\varphi+2\rho\varphi_\rho),
$$
see \eqref{P1:D0}.

Since $({\mathcal G}_1^-,\varphi)_{L^2(\cR)}=-\frac13 q_0^4(\varphi+2\rho\varphi_\rho,\varphi)_{L^2(\cR)}=0$,
the equation $ L_-u_1={\mathcal G}_1^-$ has a unique  solution $u_1$ in $L^2$
satisfying $(u_1,\varphi)_{L^2(\cR)}=0$. One can compute it explicitly:
$$
u_1=\frac16q_0^4(\rho^2\varphi+c_0\varphi),\quad c_0=-\frac{\|\rho\varphi\|_{L^2(\cR)}^2}
{\|\varphi\|_{L^2(\cR)}^2}.
$$

The solvability of the equation  $L_+v_1={\mathcal G}_1^+$ is subjected to the condition
$$
({\mathcal G}_1^+,\varphi_\rho)_{L^2(\cR)}=0,
$$
which gives
$$
q_0-Aq_0^7=0,\quad A=\frac{\|\varphi\|_{L^2(\cR)}^2}
{36\|\varphi_\rho\|_{L^2(\cR)}^2}=\frac{1}{12}.
$$
Therefore, we find
\begin{equation}\label{P1: q_0}
q_0=12^{1/6}.
\end{equation}
Under this solvability condition, the equation $L_+v_1={\mathcal G}_1^+$ has a unique solution $v_1$ satisfying $(v_1,\varphi_\rho)_{L^2(\cR)}=0$, which is given by
$$
v_1=-\frac{q_0}{3}(\partial_\rho(\rho^2\varphi)+c_1\varphi_\rho),\quad
c_1=-\frac{(\partial_\rho(\rho^2\varphi),\varphi_\rho)_{L^2(\cR)}}{\|\varphi_\rho\|_{L^2(\cR)}}.
$$
Note that $v_1$ is an odd and $u_1$ is an even function of $\rho$, in agreement with
 \eqref{P1:S1}, and they both satisfy \eqref{P1:1}.

Consider \eqref{P1:6} with $k=2$. We have
\begin{equation}
\label{P1: k=2}
L_+v_2={\mathcal G}^+_2,\quad L_-u_2={\mathcal G}^-_2,
\end{equation}
where ${\mathcal G}_2^+={\mathcal G}_2^+(\rho;q_0)$,  ${\mathcal G}_2^-={\mathcal G}^-_2(\rho;q_0)$
are smooth exponentially decaying functions of $\rho$, verifying
\eqref{P1:SD}, \eqref{P1:S2}. Since ${\mathcal G}_2^+$ is even and   ${\mathcal G}_2^-$ is odd,
the solvability conditions $({\mathcal G}_2^-,\varphi)_{L^2(\cR)}=({\mathcal G}_2^+,\varphi_\rho)_{L^2(\cR)}=0$
are  satisfied and \eqref{P1: k=2} has a unique solution $v_2$, $u_2$
with $v_2$  even and $u_2$ odd, verifying \eqref{P1:1}.

Next consider $k=3,4$.
For $k=3$ we have
\begin{equation}
\label{P1: k=3}
L_+v_3={\mathcal G}_3^+(q_0, q_1, \omega_1,\chi_1,\chi_2 ),\quad L_-u_3={\mathcal G}_3^-(q_0, q_1, \omega_1,\chi_1, \chi_2),
\end{equation}
with ${\mathcal G}_3^+(q_0,  q_1, \omega_1, \chi_1,\chi_2)$, ${\mathcal G}_3^-(q_0, q_1, \omega_1,\chi_1,\chi_2)$   verifying
\eqref{P1:SD}, \eqref{P1:S2}.
Using  \eqref{P1:D1}, \eqref{P1:D2}, \eqref{P1: q_0}, one can write ${\mathcal G}_3^+$, ${\mathcal G}_3^-$ as
\begin{equation}
\label{P1:D3}
\begin{split}
&{\mathcal G}_3^+=2(q_1+q_0\omega_1)\varphi_\rho+
4(2q_1+q_0\omega_1)\rho\varphi +\tilde {\mathcal G}_3^+ ,\\
&{\mathcal G}_3^-=-(2q_0^3q_1+\frac23 q_0^4\omega_1)
(\varphi+2\rho\varphi_\rho)-\frac23 q_0^4\omega_1(\varphi+\rho\varphi_\rho)+\tilde {\mathcal G}_3^-,
\end{split}
\end{equation}
where $\tilde {\mathcal G}_3^+=\tilde{\mathcal G}_3^+(q_0,\chi_1,\chi_2), \tilde {\mathcal G}_3^-=\tilde{\mathcal G}_3^-(q_0,\chi_1,\chi_2)$ depend on $q_0,\chi_1,\chi_2$ only, and therefore, are determined by now.

Using \eqref{P1:D3}, one can write the solvability conditions
$({\mathcal G}_3^-,\varphi)_{L^2(\cR)}=({\mathcal G}_3^+,\varphi_\rho)_{L^2(\cR)}=0$ as
$$
-\frac{q_0^4\|\varphi\|_{L^2(\cR)}^2}{3}\omega_1+(\tilde{\mathcal G}_3^-,\varphi)_{L^2(\cR)}=0,
$$
$$
2(q_1+q_0\omega_1)\|\varphi_\rho\|_{L^2(\cR)}^2-2(2q_1+q_0\omega_1)\|\varphi\|_{L^2(\cR)}^2
+(\tilde{\mathcal G}_3^+,\varphi_\rho)_{L^2(\cR)}=0,
$$
which gives,
$$
\omega_1=\frac{3(\tilde{\mathcal G}_3^-,\varphi)_{L^2(\cR)}}{q_0^4\|\varphi\|_{L^2(\cR)}^2},
$$
$$
q_1=-\frac25q_0\omega_1+\frac{(\tilde{\mathcal G}_3^+,\varphi_\rho)_{L^2(\cR)}}{10\|\varphi_\rho\|_{L^2(\cR)}^2}.
$$
With this choice of $\omega_1$, $q_1$, \eqref{P1: k=3} has a unique solution $v_3$, $u_3$
verifying \eqref{P1:1}, \eqref{P1:2} with $v_3$, $u_3$ being of the same parity as ${\mathcal G}_3^+$,
${\mathcal G}_3^-$: $v_3$ is odd and $u_3$ is even.

The case $k=4$ is similar to that of $k=2$:
we have
\begin{equation}
\label{P1: k=4}
L_+v_4={\mathcal G}_4^+,\quad L_-u_4={\mathcal G}_4^-,
\end{equation}
with ${\mathcal G}_4^+$ even and ${\mathcal G}_4^-$ odd, depending on
$q_0, q_1, \omega_1, \chi_1,\chi_2, \chi_3$ only.  Thus,
\eqref{P1: k=4} has a unique solution $v_4$, $u_4$
with $v_4$  even and $u_4$ odd, verifying
\eqref{P1:1}.

Consider the case $k=5$.
We have
\begin{equation}
\label{P1: k=5}
L_+v_5={\mathcal G}_5^+,\quad L_-u_5={\mathcal G}_5^-,
\end{equation}
with ${\mathcal G}_5^+={\mathcal G}_5^+(q_0,q_1, q_2, \omega_1, \omega_2, \chi_1,\dots, \chi_4)$ odd and
${\mathcal G}_5^-={\mathcal G}_5^-(q_0,q_1, q_2, \omega_1, \omega_2, \chi_1,\dots, \chi_4)$ even.

It follows from \eqref{P1:D2}, \eqref{P1: q_0} that
\begin{equation}
\begin{split}
&{\mathcal G}_5^+=2(q_2+q_0\omega_2)\varphi_\rho+
\frac43 (q_2+3q_0\omega_2)\rho\varphi +\tilde {\mathcal G}_5^+,\\
&{\mathcal G}_5^-=-\frac23 (4q_0^3q_2+q_0^4\omega_2)(\varphi+2\rho\varphi_\rho)-\frac43q_0^4\omega_2(\varphi+\rho\varphi_\rho)+
\tilde {\mathcal G}_5^-,
\end{split}
\end{equation}
where $\tilde {\mathcal G}_5^\pm=\tilde{\mathcal G}_5^\pm(q_0,q_1, \omega_1, \chi_1,\dots, \chi_4) $ are determined by now. Substituting  this representation into the solvability conditions
$({\mathcal G}_5^-,\varphi)_{L^2(\cR)}=0$, $({\mathcal G}_5^+,\varphi_\rho)_{L^2(\cR)}=0$, we obtain
\begin{equation}
\label{P1: S5-}
\omega_2=\frac{3(\tilde{\mathcal G}_5^-,\varphi)_{L^2(\cR)}}{2q_0^4\|\varphi\|_{L^2(\cR)}^2},
\end{equation}
\begin{equation}
\label{P1: S5+}
-4q_0\omega_2\|\varphi_\rho\|_{L^2(\cR)}^2+(\tilde{\mathcal G}_5^+,\varphi_\rho)_{L^2(\cR)}=0,
\end{equation}
which means that one has to have
\begin{equation}
\label{P1: S5}
q_0^3(\tilde{\mathcal G}_5^+,\varphi_\rho)_{L^2(\cR)}=2(\tilde{\mathcal G}_5^-,\varphi)_{L^2(\cR)}.
\end{equation}
Note that \eqref{P1: S5-}, \eqref{P1: S5+}  do not contain $q_2$, therefore, it can be chosen arbitrarily
(as soon as we can  show that \eqref{P1: S5} holds).
To check \eqref{P1: S5}, we proceed as follows. Consider
$\omega^{(1)} (t)=1+ \omega_1t^{2/3}$, $q^{(1)}(t) = q_0t^{1/3}+q_1t$, $\chi^{(1)}(\rho,t) = \sum_{k=1}^{4} t^{k/3}\chi_k(\rho)$. By our construction, they verify
\begin{equation}
\label{P1:S5:1}
\cH \vec{\chi} ^{(1)}+ l(q^{(1)},\omega^{(1)}) \vec{\chi}^{(1)} + \cF (q^{(1)},\omega^{(1)}) + F^{\geq 2} (\vec\chi^{(1)})=t^{5/3}\vec{\mathcal D}(\rho)+\frac{1}{1+\omega^{(1)}q^{(1)}\rho}{\mathcal S}(\rho,t),
\end{equation}
where $\vec{\mathcal D}={{\mathcal D}\choose -\bar{\mathcal D}}$, ${\mathcal D}=-\tilde{\mathcal G}_5^+-i\tilde{\mathcal G}_5^-$, and ${\mathcal S}(\rho,t)$  admits the estimate
\begin{equation}
\label{P1:S5:2}
|e^{\mu |\rho|}\partial_\rho^l{\mathcal S}(\rho,t)|\leq C_{l,\mu}t^{2},\quad \rho\in \cR,
\end{equation}
for any $l\geq 0$ and $\mu<1$.
Define
\begin{equation}
\label{P1:z}
z(x,t)\defeq e^{i\theta_1(t)+iv_1(t)r/2}\lambda_1(t)(\varphi+\chi^{(1)})(\lambda_1(t)(r-q^{(1)}(t)),t),
\end{equation}
where $\lambda_1$, $v_1$, $\theta_1$ are given by \eqref{E: params-1}, \eqref{E: params-2}
with $q=q^{(1)}$ and $\omega=\omega^{(1)}$.
Then $z(t)$ solves
\begin{equation}
\label{P1:S5:3}
i z_t + \Delta z+ 2 |z|^2 z = R_0(x,t)+R_1(x,t),\quad x\in \cR^3,
\end{equation}
where (by \eqref{P1:S5:1} and \eqref{P1:S5:2})
$$
R_0(x,t)=-t^{5/3}e^{i\theta_1(t)+iv_1(t)r/2}\lambda^3_1(t){\mathcal D}(\lambda_1(t)(r-q^{(1)}(t))),
$$
and $R_1$ satisfies
\begin{equation}
\label{P1:S5:4}
\|R_1\|_{L^2(\cR^3)}\leq C \, t^{2/3}.
\end{equation}
Consider the energy $E(z(t))=\int (|\nabla z(x,t)|^2 - |z(x,t)|^4)dx$.
It follows from \eqref{P1:S5:3}, \eqref{P1:S5:4}, \eqref{P1:z} that
\begin{equation}
\label{P1:S5:5}
\frac{d}{dt}E(z(t))= \alpha t^{-1}+O(t^{-2/3}),\quad t\rightarrow 0,
\end{equation}
with
\begin{equation}
\label{P1:S5:6}
\alpha=-\frac{2}{3q_0^8}\left(q_0^3(\tilde{\mathcal G}_5^+,\varphi_\rho)_{L^2(\cR)}-2(\tilde{\mathcal G}_5^-,\varphi)_{L^2(\cR)}\right).
\end{equation}
On the other hand, using the definition of $z(t)$, we can write an expansion of $E(z(t))$ in powers of $t^{1/3}$:
\begin{equation}
\label{P1:S5:7}
E(z(t))=\sum_{k=-4}^0t^{k/3}c_k +O(t^{1/3}), \quad t\rightarrow 0,
\end{equation}
with  some constants $c_k$. Comparing this expansion to \eqref{P1:S5:5} one gets
$c_{-4}=\dots=c_{-1}=0$ and $\alpha=0$, which is precisely \eqref{P1: S5}.

Verifying the solvability conditions \eqref{P1: S5-}, \eqref{P1: S5+} results in the fact that the system
\eqref{P1: k=5} has a unique solution $v_5$, $u_5$
satisfying \eqref{P1:1}, \eqref{P1:2} with $v_5$  being odd and $u_5$ being even.

To finish the proof of Proposition \ref{P:1}, we proceed by induction.
Suppose we have solved \eqref{P1:6} with $k=1,\dots, 2l-1$, $l\geq 3$, and have found
$q_j, \omega_j$, $j\leq l-1$, and $\chi_p$, $p\leq 2l-1$, verifying \eqref{P1:1}, \eqref{P1:2}, and \eqref{P1:S1}.
Consider   $k=2l, 2l+1$. For $k=2l$ we have
\begin{equation}
\label{P1: k=2l}
\begin{split}
&L_+v_{2l}={\mathcal G}_{2l}^+(q_j, \omega_j, j\leq l-1; \chi_p, p\leq 2l-1),\\
&L_-u_{2l}=
{\mathcal G}_{2l}^-(q_j, \omega_j, j\leq l-1; \chi_p, p\leq 2l-1),
\end{split}
\end{equation}
with ${\mathcal G}_{2l}^+$ being even and ${\mathcal G}_{2l}^-$ being odd,
verifying \eqref{P1:SD}.
Therefore, the equation \eqref{P1: k=2l} has a unique solution $v_{2l}$, $u_{2l}$
with $v_{2l}$  even and $u_{2l}$ odd, that satisfies \eqref{P1:1}.

Consider $k=2l+1$. We have
\begin{equation}
\label{P1: k=2l+1}
\begin{split}
&L_+v_{2l+1}={\mathcal G}_{2l+1}^+(q_j, \omega_j, j\leq l; \chi_p, p\leq 2l),\\
& L_-u_{2l+1}=
{\mathcal G}_{2l+1}^-(q_j, \omega_j, j\leq l; \chi_p, p\leq 2l),
\end{split}
\end{equation}
where ${\mathcal G}_{2l+1}^+$ is odd and ${\mathcal G}_{2l+1}^-$ is even.
From \eqref{P1:D2}, \eqref{P1: q_0} one has
\begin{equation}
\begin{split}
&{\mathcal G}_{2l+1}^+=2(q_{l}+q_0\omega_{l})\varphi_\rho -\frac43( (2l^2-l-7)q_l-3q_0\omega_{l})\rho\varphi +\tilde {\mathcal G}_{2l+1}^+\\
&{\mathcal G}_{2l+1}^-=-\frac{2q_0^3}{3} ((l+2)q_l+q_0\omega_l)(\varphi+2\rho\varphi_\rho)-\frac{2lq_0^4}{3}\omega_l(\varphi+\rho\varphi_\rho)+
\tilde {\mathcal G}_{2l+1}^-,
\end{split}
\end{equation}
where $\tilde {\mathcal G}_{2l+1}^\pm$ depends only on $ q_j, \omega_j, j\leq l-1; \chi_p, p\leq 2l$.
Using this representation one can write the solvability conditions
$({\mathcal G}_{2l+1}^-,\varphi)_{L^2(\cR)}=0$, $({\mathcal G}_{2l+1}^+,\varphi_\rho)_{L^2(\cR)}=0$ as
\begin{equation*}
\begin{split}
&-\frac{q_0^4l\omega_l}{3}\|\varphi\|_{L^2(\cR)}^2+(\tilde{\mathcal G}_{2l+1}^-,\varphi)_{L^2(\cR)}=0,\\
&2q_l(2l^2-l-6)\|\varphi_\rho\|_{L^2(\cR)}^2
-4q_0\omega_l\|\varphi_\rho\|_{L^2(\cR)}^2+(\tilde{\mathcal G}_{2l+1}^+,\varphi_\rho)_{L^2(\cR)}=0,
\end{split}
\end{equation*}
which gives
\begin{equation*}
\begin{split}
&\omega_l=\frac{3(\tilde{\mathcal G}_{2l+1}^-,\varphi)_{L^2(\cR)}}{lq_0^4\|\varphi\|_{L^2(\cR)}^2},\\
&q_l=\frac{1}{2(2l^2-l-6)}\left(\frac{4(\tilde{\mathcal G}_{2l+1}^-,\varphi)_{L^2(\cR)}}{lq_0^3\|\varphi_\rho\|_{L^2(\cR)}^2}-
\frac{(\tilde{\mathcal G}_{2l+1}^+,\varphi_\rho)_{L^2(\cR)}}{\|\varphi_\rho\|_{L^2(\cR)}^2}\right).
\end{split}
\end{equation*}
After verifying the solvability conditions, one finds $v_{2l+1}$, $u_{2l+1}$ as the unique solution of \eqref{P1: k=2l+1} satisfying \eqref{P1:2}.
\end{proof}

Define truncation functions $\theta_j \in C_c^\infty(\mathbb{R})$ by
\begin{equation}
\label{E:trunc-1}
 \theta_0(\zeta) = \left\{ \begin{array}{ll} 1, & |\zeta| \leq \frac1{200}\\ 0, & |\zeta| \geq \frac1{100} \end{array} \right. \,,
 \qquad
\theta_1 (\zeta) = \left\{
\begin{array}{ll}
1 & \text{for} ~ |\zeta| \leq \frac1{20}\\
0 & \text{for} ~ |\zeta| \geq \frac1{10}
\end{array}.
\right.
\end{equation}

For each $N \in \mathbb{N}$, introduce
\begin{equation}
\label{E:defU-N}
U^{(N)} (\rho, t) = \left(\varphi + \sum\limits_{k=1}^{2N+2} t^{k/3} \chi_k(\rho) \right) \, \theta_0(\rho \omega q).
\end{equation}
Note that
\begin{equation}
\label{E:suppU-N}
\supp U^{(N)} \subset \left\{ |\rho| \leq \frac1{100 q \omega} \right\}.
\end{equation}
By construction each approximation $U^{(N)}$ solves (see equation \eqref{E:EqU})
\begin{multline}
\label{E:EqU-2}
\qquad i \, \frac1{\lambda^2} \, U^{(N)}_t + U^{(N)}_{\rho \rho} - U^{(N)} + 2\,  |U^{(N)}|^2 U^{(N)}  = - \frac2{\rho + \lambda\, q} \, U^{(N)}_\rho\\ + \frac{q^{\prime\prime} \rho}{2 \lambda^3} \, U^{(N)}
-i \frac{\lambda^\prime}{\lambda^3} \big(1 + \rho \, \partial_\rho \big) \, U^{(N)} - i \frac{q^\prime}{\lambda (\rho + \lambda \,q)} \, U^{(N)} + \mathcal{R}_N(\rho, t) ,
\end{multline}
where the remainder $\mathcal{R}_N$ is supported in $\{ |\rho| \leq \frac1{100 q \omega}\}$ by \eqref{E:trunc-1} and \eqref{E:suppU-N}, by \eqref{P1:3} for each $N \in \mathbb{N}$ there exists a time $t_0(N)>0$ such that $\mathcal{R}_N$ satisfies the following bounds
\begin{equation}
\label{E:remR-N}
|\partial^k_\rho \mathcal{R}_N (\rho, t) | \leq C _{k}\, e^{-\mu |\rho|} \, t^{(2N+3)/3} \quad \text{for } ~ 0 < t \leq t_0(N), \quad  k = 0,1,2, ...,
\end{equation}
with the constants $\mu$ and $C_k$ independent of $N$.
\medskip

Observe also that from \eqref{E:defU-N} (the first term is time-independent $\varphi$), we have by \eqref{P1:1} that
\begin{equation}
\label{E:decayU-N}
|\partial^k_\rho U^{(N)} (\rho, t) | \leq C \, e^{-\mu |\rho|} \quad \text{for } ~ 0<t<t_0(N), \quad  k = 0,1,2, ... \, .
\end{equation}
For each $N \in \mathbb{N}$ we set $\psi^{(N)}$ by
\begin{equation}
\label{E:Psi-N}
\psi^{(N)} (r,t) = e^{i\theta(t)} e^{iv(t)r/2} \lambda(t) \, U^{(N)} \big(\lambda(t)(r-q(t)), t \big).
\end{equation}
Then $\psi^{(N)}$ solves
\begin{equation}
\label{E:eqPsi-N}
i \partial_t \psi^{(N)} + \Delta \psi^{(N)} + 2 |\psi^{(N)}|^2 \psi^{(N)} = i H^{(N)}(x,t),
\end{equation}
where
\begin{equation}
\label{E:H-N-def}
H^{(N)}(r,t) = i e^{i\theta(t)} e^{iv(t)r/2} \lambda^3 \mathcal{R}_N(\lambda(r-q(t)),t) .
\end{equation}
By \eqref{E:remR-N}, it follows that
$$
\| \partial_r^k H^{(N)} (t)\|_{L^2(\cR^3)} \leq C_k  \lambda^{3+k} t^{(2N+3)/3} \| 1_{|r-q|\leq \lambda^{-1}}\|_{L^2(\mathbb{R}^3)},
$$
and hence,
\begin{equation}
\label{E:H-N}
t^{2k/3} \| \partial_r^k H^{(N)} (t)\|_{L^2(\cR^3)} \leq C_k t^{(2N-1)/3} \quad \text{for} ~ 0 < t \leq  t_0(N), \quad k=0,1,2,... \, .
\end{equation}

\begin{proposition}
Given $e\in \mathbb{R}$, we can choose $q_2$ in the construction in Prop. \ref{P:1} such that the energy of the approximate solution $E(\psi^{(N)})$ estimates as
\begin{equation}
\label{E:E-N-t-bound-3}
|E(\psi^{(N)}) (t) -e|\leq  Ct^{\frac{2N}{3}-1} \quad \text{for} ~ 0 < t < t_0(N).
\end{equation}
\end{proposition}
\begin{proof}
The following two estimates are straightforward consequences of \eqref{E:eqPsi-N} and \eqref{E:H-N}:  \\
\begin{equation}
\label{E:Psi-N-t-bound}
\left| \frac{d}{dt} \| \psi^{(N)} (t)\|^2_{L^2(\cR^3)} \right| \leq C \, t^{\frac{2N-2}{3}} \quad \text{for} ~ 0 < t \leq t_0(N),
\end{equation}
and
\begin{equation}
\label{E:E-N-t-bound}
\left| \frac{d}{dt} E(\psi^{(N)}) (t) \right| \leq C \, \lambda^2 \, t^{\frac{2N-2}{3}} \quad \text{for} ~ 0 < t \leq t_0(N).
\end{equation}
Consider $q^{-2}(t)\lambda^{-3}(t)E(\psi^{(N)}) (t)$.
By construction of $\psi^{(N)}$, it admits an expansion in positive powers of $t^{1/3}$ (compare to \eqref{P1:S5:5}, \eqref{P1:S5:6}):
\begin{equation}
\label{E:E-N-t-bound-1}
q^{-2}(t)\lambda^{-3}(t)E(\psi^{(N)}) (t)=\sum_{k=0}^{K(N)}\beta_kt^{k/3} +O(t^\infty),\quad t\rightarrow 0,
\end{equation}
where $\beta_k$ depends on $\chi_p$, $p\leq k$, and on $\omega_j, q_j$, $j\leq k/2$. More precisely, for $k=2l$ one has
\begin{equation}
\label{E:E-N-t-bound-2}
\beta_{2l}=\frac43(\omega_l+(2l+3)\frac{q_l}{q_0})+\tilde\beta_{2l},\quad l\geq 0,
\end{equation}
with $\tilde\beta_{2l}=\tilde\beta_{2l}(\chi_p, p\leq 2l;\omega_j, q_j, j\leq l-1)$ depending on $\chi_p$, $p\leq 2l$, and on $\omega_j, q_j$, $j\leq l-1$, only.
From  \eqref{E:E-N-t-bound}, \eqref{E:E-N-t-bound-1},
we deduce $\beta_0=\dots=\beta_3=0$ and
$$
|E(\psi^{(N)}) (t) -q_0^{-4}\beta_4|\leq C t^{\frac{2N}{3}-1} \quad \text{for} ~ 0 < t \leq t_0(N).
$$
Therefore, if we fix $q_2$ by requiring
$\frac43(\omega_2+7\frac{q_2}{q_0})+\tilde\beta_{4}=eq_0^4$, then
we get \eqref{E:E-N-t-bound-3}.
\end{proof}

\section{Comparison to exact solution}
\label{S:comparison}

Let us begin with some notational conventions for this section.
We exclusively use the \emph{real} inner product
\begin{equation}
\label{E:inner-prod}
\la \psi_1, \psi_2 \ra_{L^2(\mathbb{R}^3)} = \Re \int_{\mathbb{R}^3} \psi_1(x) \overline{\psi_2(x)} \, dx
\end{equation}
on $L^2(\mathbb{R}^3)$ as well as the related 1D version
$$
\la \psi_1, \psi_2 \ra_{L^2(0<r<\infty)} = \Re \int_{r=0}^\infty \psi_1(r) \overline{\psi_2(r)} \, dr.
$$
Note that
$$
\la \psi_1,\psi_2 \ra_{L^2(\mathbb{R}^3)} = 4\pi \la r\psi_1(r),r\psi_2(r) \ra_{L^2(0<r<\infty)}.
$$

For a (densely defined) functional $A: L^2(\mathbb{R}^3) \to \mathbb{R}$, via the inner product \eqref{E:inner-prod} we identify $A'(\psi)$ with a function and $A''(\psi)$ with an operator, which is self-adjoint with respect to \eqref{E:inner-prod}.  Let
$$
J \defeq -\tfrac12 i
$$
and note that, with respect to the inner product \eqref{E:inner-prod}, $J^*=-J$.
Define the Poisson bracket
$$
\{ A, B\}(\psi) = \la JA'(\psi), B'(\psi) \ra,
$$
which yields a functional $\{A,B\}: L^2(\mathbb{R}^3) \to \mathbb{R}$.

We find it convenient to state estimates in terms of the time-dependent scaled Sobolev norms:
\begin{equation}
\label{E:X-norms}
\|h \|_{X^k(\mathbb{R}^3)} \defeq \sum_{j=0}^k t^{2j/3} \|\nabla^j h \|_{L^2(\mathbb{R}^3)}.
\end{equation}
Most frequently, we use the case $k=1$:
$$
\|h\|_{X^1(\mathbb{R}^3)} \defeq t^{2/3} \|\nabla h\|_{L^2(\mathbb{R}^3)} + \|h\|_{L^2(\mathbb{R}^3)}.
$$
By default $X^1=X^1(\mathbb{R}^3)$, although we shall have occasion to use the variant $X^1(0<r<\infty)$, which has the expected definition.

We shall frequently need the following radial Sobolev inequality.   For any radial function $f$,
\begin{equation}
\label{E:deriv-p13}
\| f \cdot 1_{r\sim t^{1/3}} \|_{L^\infty(\mathbb{R}^3)} \lesssim t^{-1/3} \|\nabla f\|_{L^2(\mathbb{R}^3)}^{1/2} \|f\|_{L^2(\mathbb{R}^3)}^{1/2}  \lesssim t^{-2/3}\|f\|_{X^1(\mathbb{R}^3)}.
\end{equation}

Recall the function $\psi^{(N)}(t)$ defined on $0< t \leq t_0(N)$ by \eqref{E:Psi-N}.
For any $0<\epsilon\leq t_0(N)$, let $\psi(t)$ be the solution to \eqref{E:NLS-1} such that
$$
\psi(\epsilon) = \psi^{(N)}(\epsilon)
$$
(so $\psi$ depends on $\epsilon$ and $N$ but this is suppressed in our notation).
Let
$$
h(t) \defeq \psi(t) - \psi^{(N)}(t).
$$
Each of the Propositions \ref{P:G}, \ref{P:G-lower}, \ref{P:kappa-bds}, \ref{P:Cauchy-N}, and \ref{P:H2-variance} in this section references $\psi$ and $h$ defined as above.  All bounds stated will be valid for $0<t\leq t_0(N)$.  All implicit constants (indicated through the notation $\lesssim$) will be independent of $\epsilon$ and $N$.

Note that \eqref{E:NLS-1} can be written as
\begin{equation}
\label{E:Ham1}
\partial_t \psi = JE'(\psi),
\end{equation}
while \eqref{E:eqPsi-N} can be written as
\begin{equation}
\label{E:Ham2}
\partial_t \psi^{(N)} = JE'(\psi^{(N)}) + H^{(N)},
\end{equation}
where $H^{(N)}$ is defined in \eqref{E:H-N-def} and satisfies \eqref{E:H-N}.

Recalling \eqref{E:trunc-1}, define the radial localized momentum
\begin{equation}
\label{E:radial-mom}
\mathcal{P}_q(\psi) = \Im \int_{\mathbb{R}^3} \left( \theta_1 \Big( \frac{r-q}{q} \Big) \right)^2 \bar \psi \; \partial_r \psi  \, dx.
\end{equation}
Define the functional
\begin{equation}
\label{E:W-def}
W(\psi) = \underbrace{\left( 1+ \frac{v^2}{4\lambda^2} \right)}_{O(1)} M(\psi) - \underbrace{\frac{v}{\lambda^2}}_{O(t^{2/3})} \mathcal{P}_q(\psi) + \underbrace{\frac{1}{\lambda^2}}_{O(t^{4/3})} E(\psi),
\end{equation}
where $M(\psi)$, $\mathcal{P}_q(\psi)$, and $E(\psi)$ are the mass, localized radial momentum, and energy functionals defined in \eqref{E:MASS}, \eqref{E:radial-mom}, and \eqref{E:ENERGY}.
Define the Lyapunov functional
\begin{equation}
\label{E:G-def}
G(h) \defeq W(\psi) - W(\psi^{(N)}) - \la W'(\psi^{(N)}),h\ra .
\end{equation}

The statement and proof of the following Lemma are based on Holmer-Lin \cite[Lemma 5.1]{HL}.
\begin{lemma}
\label{L:def-G}
For $G(h)$ as defined in \eqref{E:G-def}, we have
\begin{equation}
\label{E:deriv-p6}
\partial_t G(h) = \{E,W\}(\psi) - \{E, W\}(\psi^{(N)}) - \la \{E,W\}'(\psi^{(N)}),h\ra - \mathcal{E}_1 + \mathcal{E}_2 ,
\end{equation}
where
$$
\mathcal{E}_1 \defeq \la W''(\psi^{(N)})H^{(N)}, h \ra + \la W'(\psi^{(N)}), JE'(\psi) - JE'(\psi^{(N)}) - JE''(\psi^{(N)})h \ra,
$$
$\mathcal{E}_2$ is the result of the time derivative landing on any of the parameters in $W$, and
$H^{(N)}$ is as in \eqref{E:Ham2}.
\end{lemma}
Explicitly,
\begin{equation}
\label{E:E2}
\mathcal{E}_2 =
\begin{aligned}[t]
&( \frac12 v \dot v \lambda^{-2} - \frac12 v^2 \lambda^{-3}\dot \lambda) (M(\psi)-M(\psi^{(N)})-\la M'(\psi^{(N)}),h\ra) \\
&+ (-\dot v \lambda^{-2} - 2 v \lambda^{-3}\dot \lambda) (\mathcal{P}_q(\psi) - \mathcal{P}_q(\psi^{(N)}) - \la P_q'(\psi^{(N)}),h\ra) \\
&- \lambda^{-2}\dot \lambda (E(\psi) - E(\psi^{(N)}) - \la E'(\psi^{(N)}),h\ra) .
\end{aligned}
\end{equation}

\begin{proof}
We write expressions for the time derivatives of each term in \eqref{E:G-def}, dropping the terms that lead to \eqref{E:E2}.  First, from \eqref{E:Ham1},
\begin{equation}
\label{E:deriv-p1}
\begin{aligned}
\partial_t W(\psi) &= \la W'(\psi), \partial_t \psi\ra \\
&=   \la W'(\psi), JE'(\psi)\ra .
\end{aligned}
\end{equation}
Next, from \eqref{E:Ham2},
\begin{equation}
\label{E:deriv-p2}
\begin{aligned}
\partial_t W(\psi^{(N)}) &= \la W'(\psi^{(N)}), \partial_t \psi^{(N)}\ra \\
&= \la W'(\psi^{(N)}), JE'(\psi^{(N)}) + H^{(N)}\ra .
\end{aligned}
\end{equation}
Finally, from \eqref{E:Ham1} and \eqref{E:Ham2},
\begin{equation}
\label{E:deriv-p3}
\begin{aligned}
\partial_t \la W'(\psi^{(N)}),h \ra
&= \la W''(\psi^{(N)})\partial_t \psi^{(N)}, h \ra + \la W'(\psi^{(N)}), \partial_t \psi - \partial_t \psi^{(N)}  \ra \\
&=
\begin{aligned}[t]
&\la W''(\psi^{(N)})(JE'(\psi^{(N)})+H^{(N)}), h \ra \\
&+ \la W'(\psi^{(N)}), JE'(\psi) - JE'(\psi^{(N)}) - H^{(N)} \ra .
\end{aligned}
\end{aligned}
\end{equation}
Taking \eqref{E:deriv-p1} minus \eqref{E:deriv-p2} minus \eqref{E:deriv-p3}, and noting the cancelation of $ \la W'(\psi^{(N)}), H^{(N)}\ra$ in \eqref{E:deriv-p2} and \eqref{E:deriv-p3}, we obtain
\begin{equation}
\label{E:deriv-p4}
\partial_t G( h) =
\begin{aligned}[t]
&\la W'(\psi), JE'(\psi)\ra - \la W'(\psi^{(N)}), JE'(\psi^{(N)}) \ra \\
&- \la W''(\psi^{(N)})(JE'(\psi^{(N)})+H^{(N)}), h\ra \\
&- \la W'(\psi^{(N)}), JE'(\psi) - JE'(\psi^{(N)}) \ra .
\end{aligned}
\end{equation}
Since
$$
\{E, W\}(\psi) = \la JE'(\psi), W'(\psi)\ra ,
$$
we compute
\begin{equation}
\label{E:deriv-p5}
\la \{E,W\}'(\psi^{(N)}), h \ra = \la JE''(\psi^{(N)})h, W'(\psi^{(N)})\ra + \la JE'(\psi^{(N)}), W''(\psi^{(N)})h\ra .
\end{equation}
Now, \eqref{E:deriv-p6} follows from \eqref{E:deriv-p4} and \eqref{E:deriv-p5}.
\end{proof}

Recall the definition of the $X^k$ norm in \eqref{E:X-norms}.
By processing the terms in Lemma \ref{L:def-G}, we obtain the following upper bound on $\partial_t G$:
\begin{proposition}
\label{P:G}
Let $G(h)$ be defined as in \eqref{E:G-def} and suppose\footnote{This is part of the bootstrap hypothesis.  In fact, the argument ultimately yields the much stronger estimate $\|h\|_{X^1} \leq t^{2N/3}$.} $\|h\|_{X^1}\leq 1$.  Then
\begin{equation}
\label{E:G-h-upperbound}
|\partial_t \, G(h) | \lesssim t^{-1} \|h\|_{X^1}^2 +  t \|H^{(N)}\|_{X^1}^2 .
\end{equation}
This bound is valid for $\epsilon\leq t\leq t_0(N)$, and the implicit constants are independent of $\epsilon$ and $N$.
\end{proposition}

\begin{proof}
To apply Lemma \ref{L:def-G}, we need to bound $\mathcal{E}_1$, $\mathcal{E}_2$, and to compute the Poisson bracket $\{E, W\}$.

Computation yields (here, $\theta_1$ is short for $\theta_1((r-q)/q)$),
\begin{equation}
\label{E:Eprime}
\begin{aligned}
E'(\psi) &= -2\Delta \psi - 4|\psi|^2\psi\\
&= -2\partial_r^2 \psi - 4 r^{-1}\partial_r \psi - 4|\psi|^2\psi ,
\end{aligned}
\end{equation}
and, for convenience, taking $\alpha = \partial_r ( \theta_1^2) + 2r^{-1} \theta_1^2$, we have
\begin{equation}
\label{E:Pprime}
\mathcal{P}_q'(\psi) = -i (2\theta_1^2\partial_r \psi + \alpha \psi) .
\end{equation}

First, we address $\mathcal{E}_1$.  We have
\begin{equation}
\label{E:deriv-p11}
W'(\psi^{(N)}) =
\begin{aligned}[t]
&2\underbrace{(1+\frac14v^2\lambda^{-2})}_{\mathcal{O}(1)} \psi^{(N)}
- i\underbrace{v\lambda^{-2}}_{\mathcal{O}(t^{2/3})}( 2\theta_1^2 \partial_r \psi^{(N)} + \alpha \psi^{(N)} )\\
&+ \underbrace{\lambda^{-2}}_{\mathcal{O}(t^{4/3})} ( -2\Delta \psi^{(N)} - 4 |\psi^{(N)}|^2\psi^{(N)}) .
\end{aligned}
\end{equation}
By \eqref{E:deriv-p11} and \eqref{E:deriv-p13}, we obtain
$$
\|W'(\psi^{(N)})\|_{L^2} \lesssim \|\psi^{(N)} \|_{X^2} \lesssim 1 .
$$
However, this ignores the fact that the leading order term in $\psi^{(N)}$ is the push-forward of $\varphi(\rho)=\sech \rho$, which leads to a vanishing term in \eqref{E:deriv-p11}.  Thus, we in fact have
\begin{equation}
\label{E:deriv-p15}
\|W'(\psi^{(N)}) \|_{L^2} \lesssim t^{1/3} .
\end{equation}

We have
\begin{equation}
\label{E:deriv-p12}
E'(\psi)-E'(\psi^{(N)}) - E''(\psi^{(N)})h = \psi^{(N)}h^2 + h^3
\end{equation}
(ignoring complex conjugates).  In the pairing $\la W'(\psi^{(N)}), JE'(\psi)-JE'(\psi^{(N)}) - JE''(\psi^{(N)})h \ra$, each term will have a factor of $\psi^{(N)}$, which localizes to $r\sim t^{1/3}$.

By \eqref{E:deriv-p12}, \eqref{E:deriv-p13}, and \eqref{E:deriv-p15}, it follows that
\begin{equation}
\label{E:E1-p1}
\begin{aligned}
\indentalign |\la W'(\psi^{(N)}), JE'(\psi)-JE'(\psi^{(N)}) - JE''(\psi^{(N)})h \ra| \\
&\lesssim \|W'(\psi^{(N)})\|_{L^2} ( \|\psi^{(N)}\|_{L^\infty} \|h \cdot 1_{r\sim t^{1/3}} \|_{L^\infty} \|h\|_{L^2} + \|h\cdot 1_{r\sim t^{1/3}}\|_{L^\infty}^2 \|h\|_{L^2}) \\
&\lesssim t^{-1}\|h\|_{X^1}^2 .
\end{aligned}
\end{equation}

From \eqref{E:deriv-p11}, we find that $W''(\psi^{(N)})$ is an operator of the form
\begin{equation}
\label{E:deriv-p16}
W''(\psi^{(N)}) = \mathcal{O}(1) + \mathcal{O}(t^{2/3})(\partial_r + \alpha) + \mathcal{O}(t^{4/3})( \Delta + (\psi^{(N)})^2 \,) .
\end{equation}
In the expression, $\la W''(\psi^{(N)})H^{(N)},h \ra$, the presence of the function $H^{(N)}$ gives localization to $r\sim t^{1/3}$, and hence, \eqref{E:deriv-p13} is applicable.  It follows from \eqref{E:deriv-p16} and \eqref{E:deriv-p13} that
\begin{equation}
\label{E:E1-p2}
|\la W''(\psi^{(N)})H^{(N)}, h \ra | \lesssim \|H^{(N)}\|_{X^1} \|h\|_{X^1} \lesssim t^{-1} \|h\|_{X^1}^2 + t \|H^{(N)}\|_{X^1}^2 .
\end{equation}
Combining \eqref{E:E1-p1} and \eqref{E:E1-p2}, we obtain the bound
\begin{equation}
\label{E:E1-tot}
|\mathcal{E}_1| \lesssim t^{-1}\|h\|_{X^1}^2 + t \|H^{(N)}\|_{X^1}^2 .
\end{equation}

Next, we compute the bound on $\mathcal{E}_2$.  Working with \eqref{E:E2}, we obtain
\begin{equation}
\label{E:E2-tot}
|\mathcal{E}_2| \lesssim t^{-1} \|h\|_{X^1}^2 .
\end{equation}

Next, we compute the Poisson bracket $\{E, W\}$.
However, since $\{E,E\}=0$ and $\{E,M\}=0$, we have $\{E,W\} =  - \frac{v}{\lambda^2} \{E, \mathcal{P}_q\}$.  Substituting \eqref{E:Eprime} and \eqref{E:Pprime} and performing numerous integrations by parts, we obtain
\begin{align*}
\{E,\mathcal{P}_q\} (\psi) &= \Re \int (-i) E'(\psi) \overline{\mathcal{P}'_q(\psi)} \, dx \\
&= 8\pi \Re \int_{r=0}^\infty (-\partial_r^2 \psi - 2r^{-1}\partial_r \psi - 2|\psi|^2\psi)(2\theta_1^2 \partial_r \bar \psi  + \alpha \bar \psi) \, r^2 \, dr
\\
&=  \begin{aligned}[t]
&8\pi\int_{r=0}^\infty |\partial_r \psi|^2 \underbrace{(\partial_r (\theta_1^2r^2) -4 \theta_1^2r +\alpha r^2)}_{\sim \; t^{-1/3}r^2 \; 1_{r\sim t^{1/3}} } \, dr \\
&+ 8\pi \int_{r=0}^\infty |\psi|^4 \underbrace{( -\partial_r (\theta_1^2r^2) +2 \alpha r^2)}_{\sim \; t^{-1/3}r^2 \; 1_{r\sim t^{1/3}}} \, dr \\
&+ 8\pi \int_{r=0}^\infty |\psi|^2 \underbrace{( -\frac12 \partial_r(\alpha r^2) + \partial_r(r\alpha))}_{\sim \; t^{-1}r^2 \; 1_{r\sim t^{1/3}}} \, dr .
\end{aligned}
\end{align*}
From this, it is apparent that the quadratic (in $h$) part of $\{E,\mathcal{P}_q\}(\psi)$ satisfies
\begin{align*}
\indentalign |v\lambda^{-2}| | \{E,\mathcal{P}_q\}(\psi) - \{ E, \mathcal{P}_q\}(\psi^{(N)}) - \la \{E, \mathcal{P}_q\}'(\psi^{(N)}), h\ra | \\
&\lesssim t^{1/3} \|\nabla h\|_{L^2}^2 + t^{1/3} \|\psi^{(N)}\|_{L^\infty}^2 \|h\|_{L^2}^2 + t^{-1/3}\|h\|_{L^2}^2 \\
&\lesssim t^{-1} \|h\|_{X^1}^2 .
\end{align*}
This furnishes a bound on the main term in \eqref{E:deriv-p6}, while $\mathcal{E}_1$ and $\mathcal{E}_2$ have been bounded above in \eqref{E:E1-tot} and \eqref{E:E2-tot}.  Thus, the bound \eqref{E:G-h-upperbound} follows from Lemma \ref{L:def-G}.
\end{proof}

We next write the remainder $h$ as the push-forward of a function $f$:
\begin{equation}
\label{E:h-f}
h(r,t) = e^{i\theta(t)} e^{iv(t)r/2} \lambda(t) \, f \big(\lambda(t)(r-q(t)), t \big)
\end{equation}
and recalling that $\rho = \lambda(r-q)$, define $f_1$ as the truncation
\begin{equation}
\label{E:f-1}
f_1(\rho,t) \defeq \theta_1\left(\frac{\rho}{\lambda q}\right)\lambda^{1/2}\left( \frac{\rho}{\lambda} + q\right) f(\rho, t).
\end{equation}
Note that due to the support of $\theta_1$ (which restricts $|\rho|\ll t^{-1/3}$),
we have that $\lambda^{1/2}(\rho/\lambda + q) \sim 1$.

Let
$$
\xi_0 = i\varphi \,, \quad
\xi_1 = -\partial_\rho \varphi \,, \quad
\xi_2 = \varphi+ \rho \partial_\rho \varphi\,, \quad
\xi_3 = i\rho\varphi,
$$
where  $\varphi(\rho)=\sech\rho$,
and, recalling that $J=-\tfrac12 i$ and $J^{-1} = 2i$, define
\begin{equation}
\label{E:kappa-defs}
\kappa_j \defeq \la f_1, \tfrac12 J^{-1}\xi_j \ra_{L^2(-\infty<\rho<\infty)} .
\end{equation}
Define the projection operator
\begin{equation}
\label{E:symp-proj}
Pf_1 = -\kappa_2(f_1) \, \xi_0 + \kappa_0(f_1) \, \xi_2 - \kappa_1(f_1)\, \xi_3 + \kappa_3(f_1) \,\xi_1 .
\end{equation}

The following proposition establishes the coercivity of the functional $G(h)$.
\begin{proposition}
\label{P:G-lower}
Let $G(h)$ be defined as in \eqref{E:G-def}, and suppose that $\|h\|_{X^1} \leq t^{2/3}$. Then
\begin{equation}
\label{E:G-h-lowerbound}
\|h\|_{X^1}^2 \lesssim G(h) + \sum\limits_{i=0}^3 |\kappa_i|^2,
\end{equation}
where $\kappa_j$, $j=0,1,2,3$, are defined in \eqref{E:kappa-defs}.  This bound is valid for $\epsilon \leq t\leq t_0(N)$.  The implicit constant is independent of $\epsilon$ and $N$.
\end{proposition}

\begin{proof}
From \eqref{E:G-def},
\begin{equation}
\label{E:A2}
G(h) = \la W''(\psi^{(N)})h,h\ra + \mathcal{E}_3 ,
\end{equation}
where $\mathcal{E}_3$ takes the form
$$
\mathcal{E}_3 = \mathcal{O}(t^{4/3}) \int \psi^{(N)}h^3 + \mathcal{O}(t^{4/3}) \int h^4.
$$
Hence,
\begin{equation}
\label{E:A3}
\begin{aligned}
| \mathcal{E}_3| &\lesssim t^{4/3} \|\psi^{(N)}\|_{L^\infty} \|h \cdot 1_{r\sim t^{1/3}} \|_{L^\infty} \|h\|_{L^2}^2 + t^{4/3} \|h\|_{L^4}^4  \\
&\lesssim \| \psi^{(N)} \|_{X^1} \|h\|_{X^1}^3 + t^{4/3} \|\nabla h \|_{L^2}^2 \|h\|_{L^2} \\
&\lesssim  \|h\|_{X^1}^3 + t^{-2/3} \|h\|_{X^1}^4 \\
&\leq \frac14 \delta \|h\|_{X^1}^2,
\end{aligned}
\end{equation}
where in the second line we have applied \eqref{E:deriv-p13} in the first term and used the 3D Gagliardo-Nirenberg inequality (since there is no spatial localization to $r\sim t^{1/3}$ in this term). From \eqref{E:W-def}, we have that $W''(\psi^{(N)})$ is the operator
$$
W''(\psi^{(N)}) =
\begin{aligned}[t]
&2\left( 1 + \frac14 v^2\lambda^{-2} \right) +2i v\lambda^{-2} \theta_1^2 \partial_r \\
&+ 2\lambda^{-2} (-\partial_r^2 - 2r^{-1}\partial_r - 4|\psi^{(N)}|^2 - 2 (\psi^{(N)})^2 \mathcal{C}) ,
\end{aligned}
$$
where $\mathcal{C}$ is the operator of complex conjugation.  For convenience, let
$$
B \defeq \begin{aligned}[t]
&2\left( 1 + \frac14 v^2\lambda^{-2} \right) +2i v\lambda^{-2} \theta_1^2 \partial_r \\
&+ 2\lambda^{-2} (-\partial_r^2 - 4|\psi^{(N)}|^2 - 2 (\psi^{(N)})^2 \mathcal{C})
\end{aligned}
$$
and
$$
D \defeq 2\left( 1 + \frac14 v^2\lambda^{-2} \right)- 2\lambda^{-2}\partial_r^2 .
$$
Then using that $\partial_r^2 \, r = r\partial_r^2 + 2\partial_r$, we obtain
\begin{equation}
\label{E:A4a}
\la W''(\psi^{(N)}) h, h \ra_{L^2(\mathbb{R}^3)} = \la B (rh), rh\ra_{L^2(0\leq r \leq \infty)} .
\end{equation}
Let $\theta_2 = (1-\theta_1^2)^{1/2}$ so that $h = \theta_1^2 h + \theta_2^2 h$.  Substituting this decomposition into the right side of \eqref{E:A4a}, we get
\begin{equation}
\label{E:A4}
\la W''(\psi^{(N)}) h, h \ra_{L^2(\mathbb{R}^3)} =  \la B(r\theta_1h), r\theta_1h \ra_{L^2(0< r<\infty)} + \la D(r\theta_2h), r\theta_2h \ra_{L^2(0< r<\infty)} + \mathcal{E}_4 ,
\end{equation}
where
$$
\mathcal{E}_4 =
\sum_{j=1}^2 \big( 2\lambda^{-2}(\la 2(\partial_r \theta_j) \partial_r(\theta_jrh), rh\ra + \la (\partial_r^2\theta_j) r \theta_j h, rh\ra) \big)
$$
arises from the commutator of $\partial_r^2$ and $\theta_j$.  These terms are lower order, however, since an $r$-derivative on $h$ ``costs'' $t^{-2/3}$ in the definition of the $X^1$ norm, whereas $\partial_r$ landing on $\theta_j$ gives only the penalty of $t^{-1/3}$. Specifically,
\begin{equation}
\label{E:A5}
|\mathcal{E}_4| \lesssim t^{1/3} \|h\|_{X^1}^2 \leq \frac14 \delta \|h\|_{X^1}^2 .
\end{equation}
By \eqref{E:A4} and \eqref{E:A5},
\begin{equation}
\label{E:A8}
G(h) = \la B(r\theta_1h), r\theta_1h\ra_{L^2(0<r<\infty)} + \la D(r\theta_2h), r\theta_2h\ra_{L^2(0<r<\infty)} + \mathcal{E}_3+\mathcal{E}_4 .
\end{equation}
Define the operator
$$
A\defeq (1+\partial_\rho^2 + 4 \varphi^2 + 2\varphi^2\mathcal{C}) .
$$
Substituting \eqref{E:h-f}, \eqref{E:f-1}, we obtain
$$
\la B(r\theta_1h), r\theta_1h\ra_{L^2(0<r<\infty)} =  \la A f_1, f_1 \ra_{L^2(-\infty<\rho<\infty)} .
$$
Recalling \eqref{E:symp-proj}, it is a classical fact (see for example \cite[\S 4]{HZ}) that there exists $\delta>0$ such that
\begin{equation}
\label{E:A1}
\delta \|f_1- Pf_1\|_{L^2(\mathbb{R})}^2 \leq \la A(f_1-Pf_1),(f_1-Pf_1)\ra_{L^2(\mathbb{R})} .
\end{equation}
As a projection operator, $P$ satisfies $P^2=P$ and $P^*=P$ (adjoint with respect to the $\la \cdot, \cdot \ra$ inner product defined by \eqref{E:inner-prod}).  Hence,
$$
\|f_1-Pf_1\|_{L^2(\mathbb{R})}^2 = \|f_1\|_{L^2(\mathbb{R})}^2 - \|Pf_1\|_{L^2(\mathbb{R})}^2.
$$
Also,
$$
\|APf_1\|_{L^2(\mathbb{R})} + \|Pf_1\|_{L^2(\mathbb{R})} \lesssim \sum_{k=0}^3 |\kappa_j| .
$$
Applying this to \eqref{E:A1}, we obtain
$$
\delta \|f_1\|_{L^2(\mathbb{R})}^2 \leq \la Af_1, f_1 \ra_{L^2(\mathbb{R})}^2 + \sum_{j=0}^3 |\kappa_j|^2 .
$$
Substituting back \eqref{E:h-f}, \eqref{E:f-1}, we obtain
\begin{equation}
\label{E:A13}
\delta \|r\theta_1 h\|_{L^2(0<r<\infty)}^2 \leq \la B(r\theta_1h),(r\theta_1h)\ra_{L^2(0<r<\infty)}  + \sum_{j=0}^3 |\kappa_j|^2 .
\end{equation}
Directly from the definition of $B$, using integration by parts, we get
\begin{equation}
\label{E:A14}
t^{4/3} \|\partial_r(r\theta_1h)\|_{L^2(0<r<\infty)}^2 \lesssim \la B(r\theta_1h),(r\theta_1h)\ra_{L^2(0<r<\infty)} + \| r\theta_1 h\|_{L^2(0<r<\infty)}^2 .
\end{equation}
The inequalities \eqref{E:A13} and \eqref{E:A14} together yield
\begin{equation}
\label{E:A6}
\delta \|r\theta_1 h\|_{X^1(0<r<\infty)}^2 \leq \la B(r\theta_1h),(r\theta_1h)\ra_{L^2(0<r<\infty)}  + \sum_{j=0}^3 |\kappa_j|^2 .
\end{equation}
Directly from the definition of $D$ via integration by parts,
\begin{equation}
\label{E:A7}
\delta \|r\theta_2 h\|_{X^1(0<r<\infty)}^2 \leq \la D(r\theta_2h),(r\theta_2h)\ra_{L^2(0<r<\infty)} .
\end{equation}
Summing \eqref{E:A6} and \eqref{E:A7}, we obtain
\begin{equation}
\label{E:A9}
\begin{aligned}[t]
\indentalign \delta( \|r\theta_1h\|_{X^1(0<r<\infty)}^2 + \|r\theta_2h\|_{X^1(0<r<\infty)}^2) \\
&\lesssim  \la B(r\theta_1h),(r\theta_1h)\ra_{L^2(0<r<\infty)} + \la D(r\theta_2h),(r\theta_2h)\ra_{L^2(0<r<\infty)} + \sum_{j=0}^3 |\kappa_j|^2 .
\end{aligned}
\end{equation}
Combining \eqref{E:A8} and \eqref{E:A9},  we obtain
\begin{equation}
\label{E:A10}
\delta( \|r\theta_1h\|_{X^1(0<r<\infty)}^2 + \|r\theta_2h\|_{X^1(0<r<\infty)}^2) \\
\lesssim  G(h) + \sum_{j=0}^3 |\kappa_j|^2 - \mathcal{E}_3 - \mathcal{E}_4 .
\end{equation}
Again, by the fact that commutators are of lower order, we have
\begin{equation}
\label{E:A12}
\delta( \|r\theta_1h\|_{X^1(0<r<\infty)}^2 + \|r\theta_2h\|_{X^1(0<r<\infty)}^2) = \delta \|h \|_{X^1(\mathbb{R}^3)}^2 + \mathcal{E}_5 ,
\end{equation}
where
\begin{equation}
\label{E:A11}
|\mathcal{E}_5| \lesssim t^{1/3} \|h\|_{X^1}^2 .
\end{equation}
Combining \eqref{E:A10} and \eqref{E:A12}, and making use of error estimates \eqref{E:A3}, \eqref{E:A5}, and \eqref{E:A11}, we obtain \eqref{E:G-h-lowerbound}.
\end{proof}

We now come back to our original substitution \eqref{E:Psi-2} and the equation \eqref{E:unforced-U} for $U$ (i.e., to our solution before the conditions \eqref{E: params-1} were enforced) and the equation \eqref{E:EqU-2-mod} for $U^{(N)}$.
Consider a four parameter $(\lambda_1, \theta_1, q_1, v_1)$-family of profiles:
\begin{equation}
\label{E:phi-N}
\phi^{(N)}(\lambda_1, q_1, \theta_1, v_1, r, t) \defeq e^{i \Theta (r,t)} \tilde \lambda(t) \, U^{(N)}(\tilde \lambda(t) \big(r-\tilde q (t)\big),t)
\end{equation}
with
$$
\Theta(r,t) = \theta(t) + \theta_1 + (v(t)+v_1)r/2, \quad \tilde \lambda(t) = \lambda_1 \lambda(t), \quad \tilde q(t) = q(t) + q_1.
$$
Here, the parameters $\lambda_1$, $q_1$, $\theta_1$, and $v_1$ are assumed time-{\it independent}, and recall that the (time-dependent) parameters $\theta(t), \lambda(t), q(t),  v(t)$ satisfy conditions \eqref{E: params-1}. For the sake of brevity, we write $\rho = \tilde \lambda (r-\tilde q)$, $\tilde v = v + v_1$, and $U^{(N)}$, which stands for $U^{(N)}(\rho,t)$.
Let
\begin{equation}
\label{E:F-N}
F^{(N)}(r, t) = i \partial_t \phi^{(N)} + \Delta \phi^{(N)} + 2|\phi^{(N)}|^2\phi^{(N)}.
\end{equation}
Using the definition \eqref{E:phi-N}, the above is equivalent to 
$$
F^{(N)}(r, t) =
\begin{aligned}[t]
& e^{i \Theta (r,t)} \tilde \lambda(t) \left(i U_t^{(N)} +\tilde \lambda^2 U_{\rho\rho}^{(N)}\right) + 2 |\phi^{(N)}|^2\phi^{(N)} - \left(\dot \theta + \frac{\dot v r}2 + \frac{\tilde v^2}4 \right) \phi^{(N)}\\
& + i \, e^{i \Theta (r,t)} \left( \dot{\tilde \lambda} U^{(N)}
+\tilde v \tilde \lambda^2 U_\rho^{(N)} + \tilde \lambda (\dot{\tilde \lambda} (r-\tilde q) - \tilde \lambda \dot{\tilde q}) U_\rho^{(N)} \right)\\
& +  e^{i \Theta (r,t)} \frac{2\tilde \lambda^3}{\rho+\tilde\lambda \tilde q} \, U_\rho^{(N)}
+ i \frac{\tilde \lambda \tilde v}{\rho+\tilde \lambda \tilde q} \, \phi^{(N)}.
\end{aligned}
$$
Recalling \eqref{E:EqU-2}, multiplying it by $\tilde \lambda^3$ and substituting into
the second and third terms in the above, we obtain
\begin{equation}
\label{E:EqU-2-mod}
\begin{aligned}[t]
\indentalign F^{(N)} (r,t) =
\begin{aligned}[t]
& e^{\Theta(r,t)} {\tilde \lambda^3(t)}  \left[ i \left(\frac1{\tilde \lambda^2} - \frac1{\lambda^2} \right) U_t^{(N)}
+  \left(1 - \frac{\tilde \gamma}{\tilde \lambda^2} \right) U^{(N)} + i \frac{v_1}{\tilde \lambda} U^{(N)}_{\rho} \right.\\
& + \frac{\dot v}2 \rho  \left(\frac1{\lambda^3} - \frac1{\tilde \lambda^3} \right) U^{(N)}
+ i \frac{\dot \lambda}{\lambda} \left(\frac1{\tilde \lambda^2} - \frac1{\lambda^2} \right) \big(1 + \rho \, \partial_\rho \big) \, U^{(N)}\\
& \left. + i \left(\frac{\tilde v}{\tilde \lambda (\rho + \tilde \lambda \tilde q)} - \frac{v}{\lambda (\rho + \lambda \, q)} \right) U^{(N)} +\left(\frac{2}{\rho + \tilde \lambda \tilde q} - \frac{2}{\rho + \lambda \, q} \right) U_{\rho}^{(N)} + \mathcal{R}_N(\rho, t) \right],
\end{aligned}
\end{aligned}
\end{equation}
where $\tilde \gamma = \dot \theta + \tilde v^2/4 + \dot v \tilde q/2$ and $\rho = \tilde \lambda (r-\tilde q)$.

Define
\begin{equation}
\label{E:etas}
\begin{aligned}
\eta_0 (r,t) & = \partial_{\theta_1} \, \phi^{(N)}\Big|_{\lambda_1=1,\theta_1=q_1=v_1=0}, \\
\eta_1 (r,t) & = \frac1{\lambda(t)} \,\partial_{q_1} \phi^{(N)}\Big|_{\lambda_1=1,\theta_1=q_1=v_1=0}, \\
\eta_2 (r,t) & =  \, \partial_{\lambda_1} \phi^{(N)}\Big|_{\lambda_1=1,\theta_1=q_1=v_1=0}, \\
\eta_3 (r,t) & = (2 \lambda(t) \,\partial_{v_1} - q(t) \lambda(t) \,\partial_{\theta_1} ) \phi^{(N)}\Big|_{\lambda_1=1,\theta_1=q_1=v_1=0}.
\end{aligned}
\end{equation}

Solving for $i \partial_t \phi^{(N)}$ in \eqref{E:F-N} and using \eqref{E:EqU-2-mod} (and commutation of derivatives), we obtain
\begin{equation}
\label{E:partial-eta-0}
i \partial_t \eta_0 =  E''(\psi^{(N)}) \eta_0 +  H^{(N)},
\end{equation}
\begin{equation}
\label{E:partial-eta-1}
i \lambda^{-1} \partial_t ( \lambda \eta_1) =  E''(\psi^{(N)}) \eta_1 +i \frac12 \lambda^{-1} \dot v \, \eta_0 + O_X(t^{-2/3}),
\end{equation}
\begin{equation}
\label{E:partial-eta-2}
i \partial_t (  \eta_2 ) = E''(\psi^{(N)}) \eta_2 - 2 i \lambda^2 \, \eta_0
 + O_X(t^{-1}),
\end{equation}
\begin{equation}
\label{E:partial-eta-3}
i \lambda \, \partial_t (\lambda^{-1} \eta_3 ) = E''(\psi^{(N)}) \eta_3 - 2 i \lambda^2 \, \eta_1 + O_X(t^{-1}).
\end{equation}

\begin{proposition}
\label{P:kappa-bds}
Suppose that $\|h\|_{X^1} \leq t^{2/3}$.  There exists $\sigma_j$ satisfying
\begin{equation}
\label{E:sigma-bds}
|\kappa_j -\sigma_j| \lesssim t^{1/3}\|h\|_{X^1}
\end{equation}
so that
\begin{equation}
\label{E:sigma-est-0}
|\partial_t \sigma_0| \lesssim t^{-4/3}\|h\|_{X^1}^2 + \|H^{(N)}\|_{X^1},
\end{equation}
\begin{equation}
\label{E:sigma-est-1}
|\lambda^{-1}\partial_t \lambda \sigma_1| \lesssim  t^{-1}|\sigma_0| + t^{-2/3}\|h\|_{X^1} + \|H^{(N)}\|_{X^1},
\end{equation}
\begin{equation}
\label{E:sigma-est-2}
|\lambda \partial_t \lambda^{-1} \sigma_2| \lesssim t^{-4/3} |\sigma_0| + t^{-1} \|h\|_{X^1} + \|H^{(N)}\|_{X^1},
\end{equation}
\begin{equation}
\label{E:sigma-est-3}
|\lambda \partial_t \lambda^{-1} \sigma_3| \lesssim t^{-4/3}|\sigma_1 | + t^{-1} \|h\|_X + \|H^{(N)}\|_{X^1}.
\end{equation}
These bounds are valid for $\epsilon \leq t\leq t_0(N)$, and the implicit constants are independent of $\epsilon$ and $N$.
\end{proposition}
\begin{proof}
Let
$$
\mu(r) \defeq \theta_1((r-q)/q)\lambda^{-1/2} r^{-1}
$$
(note that $\mu(r) = O(1)$) and also let
$$
\Xi_j(r,t) \defeq  \,\mu(r) e^{i\theta} e^{ivr/2}  \lambda \, \xi_j(\lambda(r-q)).
$$
Recall $\eta_j$, $j = 0,1,2,3$, from \eqref{E:etas} and observe that
\begin{equation}
\label{E:Xi-eta}
\Xi_j = \eta_j + O_{L^2(\mathbb{R}^3)}(t^{1/3}),
\end{equation}
where the error term results from $\varphi$ being replaced by $U^{(N)}$ (see \eqref{E:defU-N}).
A change of variables calculation gives
$$
\kappa_j = \tfrac12 \la f_1, J^{-1}\xi_j\ra_{L^2(-\infty<\rho<\infty)} = \tfrac12 \la h, J^{-1} \Xi_j \ra_{L^2(\mathbb{R}^3)}.
$$
Define
\begin{equation}
\label{E:A30}
\sigma_j =  \tfrac12 \la h, J^{-1} \eta_j \ra_{L^2(\mathbb{R}^3)}.
\end{equation}
then by \eqref{E:Xi-eta} $\kappa_j-\sigma_j$ satisfies the bound \eqref{E:sigma-bds}.
Then
$$
\partial_t \sigma_0 = \tfrac12 \la \partial_t h, J^{-1}\eta_0\ra + \tfrac12 \la h, J^{-1} \partial_t \eta_0 \ra.
$$
Substituting \eqref{E:higher-en1} and \eqref{E:partial-eta-0}, we obtain
$$
\partial_t \sigma_0 =
\begin{aligned}[t]
&\tfrac12 \la J(E''(\psi^{(N)}) h + 4\psi^{(N)} |h|^2 + 2\overline{\psi^{(N)}}h^2 + 2|h|^2h) - H^{(N)}, J^{-1}\eta_0\ra \\
&+ \tfrac12 \la h, J^{-1}(JE''(\psi^{(N)})\eta_0 + H^{(N)}) \ra .
\end{aligned}
$$
Since $J^*=-J$ and $E''(\psi^{(N)})$ is self-adjoint with respect to $\la \cdot, \cdot \ra$, we have the cancelation $ \la JE''(\psi^{(N)}) h, J^{-1}\eta_0\ra + \la h, J^{-1}JE''(\psi^{(N)})\eta_0 \ra=0$.  Estimating the remaining terms, we obtain \eqref{E:sigma-est-0}.
Similarly, using the estimates \eqref{E:partial-eta-1} - \eqref{E:partial-eta-3}, we obtain the rest of \eqref{E:sigma-est-1}-\eqref{E:sigma-est-3}.
\end{proof}

\begin{proposition}
\label{P:Cauchy-N}
For a sufficiently large $N$, independent of $0<\epsilon<t_0(N)$, we have
\begin{equation}
\label{E:bs-reinforced}
\|h\|_{X^1} \leq t^{2N/3}, \quad \epsilon \leq  t \leq t_0(N).
\end{equation}
\end{proposition}

\begin{proof}
This will be proved by a bootstrap argument invoking the estimates obtained in Props. \ref{P:G}, \ref{P:G-lower}, and \ref{P:kappa-bds}.  Recall that
\begin{equation}
\label{E:HN-bd}
\|H^{(N)}\|_{X^1} \leq t^{(2N-1)/3} .
\end{equation}

We make the following bootstrap assumption:
\begin{equation}
\label{E:bs}
\forall \; 0<t\leq t_0(N), \qquad \|h\|_{X^1} \leq 2t^{2N/3} .
\end{equation}

Assumption \eqref{E:bs} will be validated provided we can show, using Propositions \ref{P:G}, \ref{P:G-lower}, and \ref{P:kappa-bds}, that \eqref{E:bs} reinforces itself -- specifically that \eqref{E:bs-reinforced}
holds as a consequence.

By \eqref{E:bs}, \eqref{E:HN-bd} inserted into \eqref{E:G-h-upperbound}
$$
|\partial_t G(h) | \lesssim t^{4N/3-1} .
$$
Integrating, we obtain
\begin{equation}
\label{E:A41}
G(h) = G(h(t))-G(h(\epsilon)) \lesssim \frac{1}{N} t^{4N/3} .
\end{equation}

By \eqref{E:bs}, \eqref{E:HN-bd}, inserted into \eqref{E:sigma-est-0},
$$
|\partial_t \sigma_0|  \lesssim t^{(2N-1)/3}.
$$
Integrating over $[\epsilon,t]$ using that $\sigma_j(\epsilon)=0$, we obtain
\begin{equation}
\label{E:A40}
|\sigma_0| \lesssim t^{(2N+2)/3} .
\end{equation}
By \eqref{E:bs}, \eqref{E:HN-bd}, \eqref{E:A40} inserted into \eqref{E:sigma-est-1} and \eqref{E:sigma-est-2}, we obtain
$$
| \lambda^{-1} \partial_t \lambda \sigma_1| \lesssim t^{(2N-2)/3} \,, \qquad
| \lambda \partial_t \lambda^{-1} \sigma_2 | \lesssim t^{(2N-3)/3} .
$$
Integrating, we obtain
\begin{equation}
\label{E:A40B}
|\sigma_1| \lesssim t^{(2N+1)/3} \,, \qquad |\sigma_2| \lesssim \frac{1}{N} t^{2N/3} .
\end{equation}
Inserting \eqref{E:bs}, \eqref{E:HN-bd}, \eqref{E:A40B} into \eqref{E:sigma-est-3}, we obtain
$$
|\lambda \partial_t \lambda^{-1} \sigma_3| \lesssim t^{(2N-3)/3} .
$$
Integrating, we obtain
\begin{equation}
\label{E:A40C}
| \sigma_3| \lesssim \frac{1}{N} t^{2N/3} .
\end{equation}
By \eqref{E:sigma-bds}, $\kappa_j$ can be replaced by $\sigma_j$ in \eqref{E:G-h-lowerbound}.  By  \eqref{E:A41}, \eqref{E:A40}, \eqref{E:A40B}, \eqref{E:A40C} inserted into  \eqref{E:G-h-lowerbound},
$$
\|h\|_{X^1}^2 \lesssim \frac{1}{N} t^{4N/3},
$$
and hence, \eqref{E:bs-reinforced} holds by taking $N$ sufficiently large.

\end{proof}

In order to construct the solution advertised in Theorem \ref{T:MainTheorem}, we will carry out a compactness argument.  For this compactness argument, we need higher regularity ($H^2$) control and tighter localization:

\begin{proposition}
For $N$ taken large enough so that Prop. \ref{P:Cauchy-N} holds, we have the further bounds
\label{P:H2-variance}
\begin{equation}
\label{E:Psi-H2}
t^{4/3}\|\nabla^2 h \|_{L^2(\mathbb{R}^3)} \lesssim t^{\frac{2N}{3}-1}
\end{equation}
and
\begin{equation}
\label{E:Psi-var}
\|x\psi(t)\|_{L^2(\mathbb{R}^3)} \lesssim  t^{1/3}
\end{equation}
for $\epsilon \leq t\leq t_0(N)$, with implicit constants independent of $N$ and $0<\epsilon \leq t_0(N)$.
\end{proposition}

\begin{proof}
Recall
\begin{equation}
\label{E:higher-en1}
\partial_t h =J(E''(\psi^{(N)}) h + 8\psi^{(N)} |h|^2 + 4\overline{\psi^{(N)}}h^2 + 4|h|^2h) - H^{(N)},
\end{equation}
where
$$
JE''(\psi^{(N)}) = +i(\Delta + 4 |\psi^{(N)}|^2 +2 (\psi^{(N)})^2 \mathcal{C})
$$
and $\mathcal{C}$ denotes the operator of complex conjugation.  We have available $H_x^1$ control of $h$ and thus seek a ``higher-order energy estimate''.
Applying $\partial_j$ ($j=1,2,3$) to \eqref{E:higher-en1}, we obtain
\begin{equation}
\label{E:higher-en25}
\partial_t \partial_j h = JE''(\psi^{(N)}) \partial_j h +  A(h) + B(h,\partial_j h) - \partial_j H^{(N)},
\end{equation}
where $A(h)$ contains terms of the form $\psi_N \cdot \partial_j\psi_N \cdot h$ (up to complex conjugation) and $B(h,\partial_j h)$ denotes terms that contain one power of $\partial_j h$ or $\partial_j \bar h$ and at least one power of $h$.   Thus $\partial_j h$, modulo error terms, satisfies at the linear level the same equation satisfied by $h$ itself.
Let
$$\tilde \psi \defeq \psi^{(N)}+ \partial_j h \,.$$
Then by \eqref{E:Ham2} and \eqref{E:higher-en25},
\begin{equation}
\label{E:higher-en21}
\partial_t \tilde \psi = JE'(\psi^{(N)}) + JE''(\psi^{(N)}) \partial_j h +  A(h) + B(h,\partial_j h) + H^{(N)} - \partial_j H^{(N)}.
\end{equation}
Note the expansion
\begin{equation}
\label{E:higher-en26}
JE'(\tilde \psi) = JE'(\psi^{(N)}+ \partial_jh) = JE'(\psi^{(N)}) + JE''(\psi^{(N)})\partial_j h + JF,
\end{equation}
where
\begin{align*}
F &\defeq E'(\tilde \psi) - E'(\psi^{(N)}) - E''(\psi^{(N)})\partial_j h \\
&=4\psi^{(N)}|\partial_jh|^2 + 2 \overline{\psi^{(N)}} (\partial_j h)^2 + 2|\partial_j h|^2 \partial_j h.
\end{align*}
Using the expansion \eqref{E:higher-en26}, the equation \eqref{E:higher-en21} becomes
\begin{equation}
\label{E:higher-en20}
\partial_t \tilde \psi = JE'(\tilde \psi)  +D -JF,
\end{equation}
where
$$
D \defeq A(h) + B(h,\partial_j h) +  H^{(N)} - \partial_j H^{(N)}.
$$
We now use the functional $G$ that appeared in Lemma \ref{L:def-G} and Prop. \ref{P:G}, \ref{P:G-lower}, where we replace $\psi$ by $\tilde \psi$.  Specifically, in place of \eqref{E:G-def} we take
$$
G(\partial_j h) \defeq W( \tilde \psi) - W(\psi^{(N)}) - \la W'(\psi^{(N)}),\partial_j h \ra.
$$
A slightly modified version of Lemma \ref{L:def-G} follows, in which \eqref{E:higher-en20} is applied in place of \eqref{E:Ham1}, and the following identity is obtained in place of \eqref{E:deriv-p6}
$$
\partial_t G(\partial_j h) = \{E,W\}(\tilde \psi) - \{E,W\}(\psi^{(N)}) - \la \{E,W\}'(\psi^{(N)}), \partial_j h \ra - \mathcal{E}_1 + \mathcal{E}_2,
$$
where
$$
\mathcal{E}_1 =
\begin{aligned}[t]
&\la W''(\psi^{(N)})H^{(N)}, \partial_j h\ra - \la W'(\tilde\psi) - W'(\psi^{(N)}), D-JF\ra \\
&+ \la W'(\psi^{(N)}), JF\ra
\end{aligned}
$$
and $\mathcal{E}_2$ is the same as in \eqref{E:E2}, but with $\psi$ replaced by $\tilde \psi$.
The proof of Prop. \ref{P:G} modifies accordingly, to yield in place of \eqref{E:G-h-upperbound}
\begin{equation}
\label{E:higher-en23}
| \partial_t G(\partial_j h) | \lesssim t^{-1} \|\partial_j h \|_{X^1}^2 +  t^{\frac{4N}{3}-3}
\end{equation}
Indeed, $\mathcal{E}_1$ is estimated using $\|W'(\psi^{(N)})\|_{L^2} \lesssim t^{1/3}$, and
$$\|F\|_{L^2} + t^{2/3} \| 1_{r\sim t^{1/3}} \nabla F\|_{L^2} \lesssim t^{-4/3} \| \partial_j h \|_{X^1}^2$$
$$\|D\|_{L^2} + t^{2/3} \| 1_{r\sim t^{1/3}} \nabla D \|_{L^2} \lesssim t^{-2} \|h\|_{X^1} + t^{-4/3} \|h\|_{X^1} \|\partial_j h \|_{X^1}$$
By an argument following the proof of Prop. \ref{P:G-lower},
\begin{equation}
\label{E:higher-en3}
\| \partial_j h \|_{X^1}^2 \lesssim G(\partial_j h) + \sum_{i=0}^3 |\tilde \kappa_i|^2,
\end{equation}
where $\tilde \kappa_i = \la \partial_j h, \frac12 J^{-1} \Xi_i\ra$ and $\Xi_i$ is given in \eqref{E:Xi-eta}.  The equations \eqref{E:higher-en23} and \eqref{E:higher-en3} can be combined to yield \eqref{E:Psi-H2}.

Next, we will establish \eqref{E:Psi-var}.  The pseudoconformal conservation law (see Strauss \cite[p. 13]{Strauss}) is
\begin{equation}
\label{E:A46}
\partial_t \left( \frac12 \|(x+2it\nabla) \psi\|_{L^2(\mathbb{R}^3)}^2 -2t^2 \|\psi\|_{L^4(\mathbb{R}^3)}^4\right) = 2t \|\psi\|_{L^4(\mathbb{R}^3)}^4 .
\end{equation}
By energy conservation,
\begin{equation}
\label{E:A45}
\|\psi\|_{L^4(\mathbb{R}^3)}^4 = \|\nabla \psi\|_{L^2(\mathbb{R}^3)}^2 - E(\psi) \lesssim t^{-4/3} .
\end{equation}
By  \eqref{E:E-N-t-bound-3}, $E(\psi) = E(\psi^{(N)}(\epsilon)) \sim 1$ (independently of $\epsilon$), and hence, \eqref{E:A45} yields
$$
\| \psi\|_{L^4(\mathbb{R}^3)} \lesssim t^{-1/3} .
$$
Applying this to \eqref{E:A46}, we obtain
$$
\|(x+2it\nabla) \psi\|_{L^2(\mathbb{R}^3)} \lesssim \|(x+2it\nabla) \psi^{(N)} \big|_{t=\epsilon} \|_{L^2(\mathbb{R}^3)} + t^{1/3}.
$$
By the localization of $\psi^{(N)}$ to $|x|\sim t^{1/3}$, we have
$$
\|(x+2it\nabla)\psi^{(N)}(t)\|_{L^2} \sim t^{1/3},
$$
from which \eqref{E:Psi-var} follows.
\end{proof}

\section{Proof of Theorem \ref{T:MainTheorem}}

Fix $N$ sufficiently large so that Propositions \ref{P:Cauchy-N} and \ref{P:H2-variance} are applicable.  Let $t_0=t_0(N)$.  For any sequence $\epsilon_j \to 0$, let $\psi_j$ be the solution to \eqref{E:NLS-1} as described at the beginning of \S \ref{S:comparison} with $\epsilon=\epsilon_j$.  By Proposition \ref{P:Cauchy-N} and \ref{P:H2-variance}, we have
$$
\|\psi_j(t_0)\|_{X^2(\mathbb{R}^3)} + t^{-1/3}\|x\psi_j(t_0)\|_{L^2} \lesssim 1
$$
with implicit constant independent of $j$.  By the Rellich compactness theorem, we can pass to a subsequence such that $\psi_j(t_0) \to \psi_0$ in $X^1(\mathbb{R}^3)$.  Let $\psi$ be the solution to \eqref{E:NLS-1} with $\psi(t_0)=\psi_0$.  By Proposition \ref{P:Cauchy-N}
\begin{equation}
\label{E:A47}
\sup_{\epsilon_j \leq t\leq t_0} \|\psi_j(t)-\psi^{(N)}(t)\|_{X^1(\mathbb{R}^3)} \leq t^{2N/3} .
\end{equation}
By continuous dependence on initial conditions in the  Cauchy problem, for each fixed $0<t\leq t_0$, we have $\psi_j(t) \to \psi(t)$ in $H^1(\mathbb{R}^3)$ as $j\to \infty$.  Hence, we can send $j\to \infty$ in \eqref{E:A47} to obtain
$$
\sup_{0<t\leq t_0} \|\psi(t)-\psi^{(N)}(t)\|_{X^1(\mathbb{R}^3)} \leq t^{2N/3} .
$$
This yields the bound on the first two terms in \eqref{E:h-est}.  To deduce the bound on the third term in \eqref{E:h-est}, we follow the argument in the proof of Proposition \ref{P:H2-variance} utilizing the pseudoconformal conservation law.


\begin{thebibliography}{00}
\bibitem{BW}
J. Bourgain, W. Wang,
\emph{Construction of blowup solutions for the nonlinear Schr\"odinger equation with critical nonlinearity},
Annali della Scuola Normale Superiore di Pisa - Classe di Scienze, S\'er. 4, 25 no. 1-2 (1997), p. 197--215.

\bibitem{DZR}
L.M. Degtyarev, V.E. Zakharov, L.V. Rudakov,
\emph{Two examples of Langmuir wave collapse},
Sov. Phys. JETP 41 (1975), pp. 57--61.

\bibitem{FGW}
G. Fibich, N. Gavish, X.-P. Wang,
\emph{Singular ring solutions of critical and supercritical nonlinear Schr\"odinger equations},
Phys. D 231 (2007), no. 1, pp. 55--86.

\bibitem{HL}
J. Holmer and Q.-H. Lin,
\emph{Phase-driven interaction of widely separated nonlinear Schr\"odinger solitons}, to appear in Journal of Hyperbolic Differential Equations, available at \texttt{arXiv:1108.4859 [math.AP]}.

\bibitem{AMRX}
J. Holmer and S. Roudenko,
\emph{On blow-up solutions to the 3D cubic nonlinear Schr\"odinger equation},
Appl. Math. Res. Express. AMRX 2007, no. 1, Art. ID abm004, 31 pp.

\bibitem{HR-above}
J. Holmer, R. Platte, and S. Roudenko, \emph{Blow-up criteria for the 3D cubic nonlinear Schr\"odinger equation}, Nonlinearity 23 (2010), no. 4, pp. 977--1030.

\bibitem{HZ}
J. Holmer and M. Zworski, \emph{Slow soliton interaction with delta impurities}, J. Mod. Dyn. 1 (2007), no. 4, pp. 689--718.

\bibitem{MRS}
F. Merle, P. Rapha\"el, J. Szeftel,
\emph{On collapsing ring blow up solutions to the mass supercritical NLS},
\texttt{arXiv:1202.5218 [math.AP]}.

\bibitem{G2011}
Galina Perelman, Analysis seminar, Universit\'e de Cergy-Pontoise, Dec 2011 (joint work with J.
Holmer and S. Roudenko).

\bibitem{Strauss}
W. Strauss, \emph{Nonlinear wave equations}, CBMS Regional Conference Series in Mathematics, 73. Published for the Conference Board of the Mathematical Sciences, Washington, DC; by the American Mathematical Society, Providence, RI, 1989. x+91 pp. ISBN: 0-8218-0725-0

\bibitem{Tao-book}
T. Tao,
\emph{Nonlinear dispersive equations. Local and global analysis}, CBMS Regional Conference Series in Mathematics, 106. Published for the Conference Board of the Mathematical Sciences, Washington, DC; by the American Mathematical Society, Providence, RI, 2006. xvi+373 pp. ISBN: 0-8218-4143-2.

\bibitem{Wei}
M.I. Weinstein,
\emph{Lyapunov stability of ground states of nonlinear dispersive evolution equations}, Comm. Pure Appl. Math. 39 (1986), no. 1, pp. 51--67.

\bibitem{SS}
C. Sulem and  P.-L. Sulem,
\emph{The nonlinear Schr\"odinger equation. Self-focusing and wave collapse}.
Applied Mathematical Sciences, 139. Springer-Verlag, New York, 1999. xvi+350 pp.
\end{thebibliography}
\end{document}